\documentclass[11pt]{article}
\usepackage{amsmath, latexsym, amsfonts, amssymb, amsthm, amscd}
\usepackage{color}
\usepackage{amsmath,amssymb,amsthm, color}
\usepackage{version,tabularx,multicol}
\usepackage{graphicx,float, url, dsfont}
\usepackage{pdfsync}
\usepackage{color}
\usepackage{graphicx}
\usepackage{epsfig}
\usepackage{subfigure}
\usepackage{color}
\usepackage{float}
\usepackage{color}
\usepackage{flafter}

\usepackage[utf8]{inputenc}

\newtheorem{theorem}{Theorem}
\newtheorem{corollary}{Corollary}
\newtheorem{proposition}{Proposition}
\newtheorem{lemma}{Lemma}

\newtheorem{remark}{Remark}
\theoremstyle{definition}

\newtheorem{definition}{Definition}
\newtheorem{example}{Example}
\newcommand{\p}{\Bbb{P}}

\newcommand{\e}{\mathrm e}

\newcommand{\N}{\mathbb N}
\newcommand{\R}{\mathbb{R}}

\newcommand{\dd}{\mathrm{d}}

\newcommand{\E}{\ensuremath{\mathbb{E}}}

\renewcommand{\theequation}{\thesection.\arabic{equation}}

\newcommand\numberthis{\addtocounter{equation}{1}\tag{\theequation}}
\renewcommand{\P}{\mathbb P}

\usepackage[T1]{fontenc}
\usepackage[utf8]{inputenc}
\usepackage{authblk}

\title{Particle systems with coordination}
\author[1]{Adri\'an Gonz\'alez Casanova\footnote{Universidad Nacional Aut\'onoma de M\'exico, IMATE campus CU, 04510, M\'exico, adriangcs@matem.unam.mx}, Noemi Kurt\footnote{TU Berlin, Sekr. MA 7-5, Stra\ss e des 17. Juni 136, D-10623 Berlin, kurt@math.tu-berlin.de}, Andr\'as T\'obi\'as\footnote{TU Berlin, Sekr. MA 7-5, Stra\ss e des 17. Juni 136, D-10623 Berlin, tobias@math.tu-berlin.de}}

\begin{document}
\maketitle 
\begin{abstract}
We consider a generalization of spatial branching coalescing processes in which the behaviour of individuals is not (necessarily) independent, on the contrary, individuals tend to take simultaneous actions. We show that these processes have moment duals, which happen to be multidimensional diffusions with jumps. Moment duality provides a general framework to study structural properties of the processes in this class. We present some conditions under which the expectation of the process is not affected by coordination and comment on the effect of coordination on the variance. We analyse several examples in more detail, including the nested coalescent, the peripatric coalescent with selection and coordinated migration, and the Parabolic Anderson Model.
\end{abstract}



\section{Introduction}

Spatial branching coalescing processes and their duals have received considerable attention in the literature. For example, in \cite{AS} particle systems on a lattice are considered, where particles undergo migration, death, branching and (pair) coalescence, independently of one another. These processes are dual to certain interacting diffusions used in the modelling of spatially interacting populations with mutation, selection and resampling. One of the many questions of investigation is the long time behaviour of such processes.

On the other hand, \emph{coordinated} transitions of several particles have lead to interesting processes in a number of models which are already well-studied in the literature, such as multiple merger coalescents \cite{pitman, MS}. More recently, the seed-bank coalescent with simultaneous switching \cite{BGKW18} has shown qualitatively different features compared to its non-coordinated version. In both these cases, the effects of coordinated vs.~independent actions of particles may lead to drastically different long term behaviour, reflected for example in the question of \lq coming down from infinity\rq.

In this paper, we present a unified framework of spatial branching coalescing interacting particle systems, where all types of occurring transitions may occur in a coordinated manner. For simplicity, we restrict our presentation to the case of finitely many spatial locations, except for a few remarks in the last section. In our construction, the size of a coordinated transition is determined according to a measure on $[0,1]$. The individuals then \lq decide\rq\ independently according to the size of the transition whether or not to participate. This construction is reminiscent of Lambda-coalescents or of the seed-bank coalescent with simultaneous switching, and leads (under some suitable conditions on the measures involved) to a continuous time Markov chain with finite jump rates. 

As a first result in Section \ref{sec1}, we prove moment duality for this class of coordinated processes. The
SDEs arings as dual processes will then be interpreted in terms of population genetics. We discuss a number of examples of processes in the (recent) mathematical literature, and construct their duals, some of which seem to be new. Via duality, we also provide some results on the long time behaviour of these models, such as a criterion for coming down from infinity for the so-called nested coalescent \cite{BDLS} and for models exhibiting death but no coalescence, and almost sure fixation in a variant of the peripatric coalescent \cite{LM} with non-coordinated selection and coordinated migration. Examples extend to situations not generally looked at from the point of view of population models, such as the famous Parabolic Anderson Model (PAM).

In Section~\ref{sec-E} we show that in absence of coalescence, the expectation of the coordinated branching coalescing process is the unique solution of a system of ODEs depending only on the total mass of the defining measures. As an example, we consider the PAM branching process and provide a straightforward new proof of the well-known Feynman-Kac formula based on our observation. In Section~\ref{sec-Var}, also in the coalescence-free case, we identify the choice of reproduction, death and migration measures that maximize or minimize the variance of the process, given the total masses of these measures. We use this to provide an upper bound on the variance of the PAM branching process. In Section~\ref{sec-infinity}, we extend some of our results, in particular our proof for the Feynman-Kac formula, to a class of infinite graphs. We point out that a prominent example of a process of our class on infinite graphs is the binary contact path process \cite{Griffeath}, a simple function of which is the contact process. Finally, as another application of the invariance of expectation, we provide a probabilistic interpretation of the expectation process of a branching random walk on an infinite uniform rooted tree.

\section{Coordinated branching coalescing processes}\label{sec-modeldef}
In this section, we present the general framework of this paper. Without coordination,
spatial branching coalescing particle systems in the setup of \cite{AS} are continuous time Markov chains with transitions according to the following definition:

\begin{definition}\label{def:fixed_rate}
Consider a finite set $V$. We write $e_v$ for the unit vector with 1 at the $v$-th coordinate, $v\in V.$ For each $v\in V$ fix the following parameters for
\begin{itemize} 
\item pair-coalescence: $c_v\geq 0$,
\item death: $d_v\geq 0$,
\end{itemize}
and for each pair $(v,u)\in V\times V$ fix parameters for
\begin{itemize}
\item branching: $r_{vu}\geq 0$, and
\item migration: $m_{vu}\geq 0$.
\end{itemize}
A \emph{structured branching coalescing process} on $V$ with these parameters is the continuous time Markov chain $(Z_t)_{t\geq 0}$ taking values in $\N_0^V$ with the following transitions:
\begin{equation}\label{eq:jumps_fixed}
z \mapsto 
\begin{cases}
z- e_v+ e_u, & \textrm{ at rate }z_v m_{vu},  v,u\in V \\
z+ e_u, & \textrm{ at rate }z_v r_{vu},  v,u\in V \\
z- e_v, & \textrm{ at rate } c_v\binom{z_v}{2}+d_vz_v.
\end{cases}
\end{equation}
\end{definition}
The set $V$ may be chosen countably infinite, see \cite{AS}, but for this paper we will restrict ourselves to finite sets. We assume $m_{vv}=0.$ Binary branching may be extended to more general reproduction mechanisms as in \cite{GCSW}, and more general pairwise interactions, see \cite{GCPP}. 

In order to include coordination into such models, we replace the positive real-valued parameters of Definition \ref{def:fixed_rate} with measures on $[0,1]$. Denote by $ \mathcal{M}[0,1]$ the space of finite measures on $[0,1]$. 

\begin{definition}\label{def:random_rate}
Fix a finite set $V$. For each $v\in V$ fix measures
\begin{itemize}
\item coalescence: $\Lambda_v\in \mathcal{M}[0,1]$,
\item death: $D_v\in \mathcal{M}[0,1]$,
\end{itemize}
and for each pair $(v,u)\in V\times V$ fix measures
\begin{itemize}
\item reproduction: $R_{vu}\in \mathcal{M}[0,1]$,
\item migration: $M_{vu}\in \mathcal{M}[0,1]$.
\end{itemize}
A \emph{structured branching coalescing process with coordination} with these parameters is the continuous time Markov chain $(Z_t)_{t\geq 0}=(Z_t^{(v)})_{t \geq 0, v \in V}$ with values in the set $\N_0^V$ of functions having domain $V$ and values in $\N_0$ such that
\begin{equation}\label{eq:jumps_random}
z \mapsto 
\begin{cases}
z- ie_v+ ie_u, & \textrm{ at rate }\int_0^1\binom{z_v}{i}y^i(1-y)^{z_v-i}\frac{1}{y}M_{vu}(\dd y), \, u,v\in V, 1\leq i\leq z_v\\
z- ie_v, & \textrm{ at rate }  \int_0^1\binom{z_v}{i}y^i(1-y)^{z_v-i}\frac{1}{y}D_{v}(\dd y),\, v\in V, 1\leq i\leq z_v\\
z+ ie_u, & \textrm{ at rate } \int_0^1\binom{z_v}{i}y^i(1-y)^{z_v-i}\frac{1}{y}R_{vu}(\dd y),\, u,v\in V, 1\leq i\leq z_v\\
z- (i-1)e_v, & \textrm{ at rate } \int_0^1\binom{z_v}{i}y^i(1-y)^{z_v-i}\frac{1}{y^2}\Lambda_{v}(\dd y),\, v\in V, 2\leq i\leq z_v.
\end{cases}
\end{equation}
\end{definition}
The rates of this process may be interpreted as individuals deciding independently to participate in an event, leading to a binomial number of affected individuals. The probability to participate in, say, a migration event from $v$ to $u$ is determined by the measure $y^{-1}M_{vu}(\dd y)$. This measure has a singularity at $y=0$ and is not necessarily finite, but (with only slight abuse of notation) \[ \int_{\{0\}}\binom{z_v}{i}y^i(1-y)^{z_v-i}\frac{1}{y}M_{vu}(\dd y)=z_v\mathds 1_{\{i=1\}}M_{vu}(\{0\}) \numberthis\label{delta0} \] 
is finite, analogously for the death and reproduction. For the coalescence, similarly,
\[ \int_{\{0\}}\binom{z_v}{i}y^i(1-y)^{z_v-i}\frac{1}{y^2}\Lambda_{v}(\dd y)=\binom{z_v}{2}\mathds 1_{\{i=2\}}\Lambda_{v}(\{0\}). \numberthis\label{delta0Lambda} \]
We will further discuss the role of these singularities below.

We abbreviate the total masses of $\Lambda_v,D_v,R_{vu}$ and $M_{vu}$ by $c_v$, $d_v$, $r_{vu}$ and $m_{vu}$ respectively, where we again assume that $m_{vv}=0$. Throughout the paper, $\delta_a$ denotes the Dirac measure in $a\in[0,1]$. Thus, in terms of Definition~\ref{def:random_rate}, the process defined in Definition~\ref{def:fixed_rate} corresponds to each of these measures being equal to the corresponding total mass times $\delta_0$. We refer to the points $v \in V$ as \emph{vertices}, which we often interpret as \emph{islands} in a spatial population model with multiple islands. Another possible interpretation of $V$ may be that of a type space, leading to multitype branching coalescing processes with coordination. The (undirected) \emph{interaction graph} associated to the measures defined in Definition~\ref{def:random_rate} is given as $G=(V,E)$, where
\[ E=\{ (u,v) \in V \times V \colon u \neq v \text{ and } \max \{ r_{uv}, r_{vu}, m_{uv}, m_{vu} \}>0 \}, \]
i.e., the vertex set of $G$ equals $V$ and we connect two different vertices $u,v \in V$ by an edge whenever there is interaction by migration or reproduction between $u$ and $v$.

We prove in Lemma \ref{lem:markov} below that \eqref{eq:jumps_random} indeed yields a Markov chain. Its infinitesimal generator is expressed below in \eqref{eq:generator}. 

Coordination of interactions in the above sense may be interpreted by means of suitable Poisson processes. We illustrate this by a first example. 

\begin{example}\label{ex:bpre}
Consider the non-spatial case $V=\{ 1 \}$ without migration, and with $\Lambda_1=D_1=0$ fixed. We first let $R_1=r_1\delta_0$ for some $r>0$. Then the rate for a branching event producing $i$ offspring if there are presently $z$ particles is given by
\[r_1\int_0^1 \binom{z}{i}y^{i-1}(1-y)^{z-i}\delta_0(\dd y)=r_1z\mathds 1_{\{i=1\}},\]
thus $(Z_t)_{t\geq 0}$ is a binary branching process where particles reproduce independently at rate $r_1,$ i.e.\ a Yule process. If on the other hand $R_1=r_1\delta_1$, then the reproduction rate is given by
\[r_1\int_0^1 \binom{z}{i}y^{i-1}(1-y)^{z-i}\delta_1(\dd y)=r_1\mathds 1_{\{i=z\}}.\]
In this case we may look at the process from the following viewpoint: Reproduction events happen according to a Poisson process with intensity $r_1$, and at each arrival time of the Poisson process, every particle produces exactly one new particle. That is, the resulting process is given by $(2^{N_t^{(r_1)}})_{t\geq 0}$ where for $\lambda>0$, $(N^{(\lambda)}_t)_{t\geq 0}$ denotes a Poisson process with intensity $\lambda$. The main difference between the two cases considered here is independence vs.~coordination. The Dirac measure $R_1=\delta_0$ gives full independence, while for $R_1=\delta_1$ the reproduction events are fully coordinated. The choice $R_1=r_1\delta_w$ for some $w\in (0,1)$ leads to a model in which reproduction events arrive according to a Poisson process with intensity $r_1/w$ and at each reproduction event each individual reproduces with probability $w$. It is interesting to observe that in all these cases $$\E_n[Z_t]=n\E\Big[(1+w)^{N^{(r_1/w)}_t}\Big]=\mathrm e^{r_1t}. $$
As we will see in Lemma \ref{lem:expectationequality}, the invariance of the expectation is not a coincidence. In the general case $R_{1}\in \mathcal{M}[0,1]$, the process will also have expectation $\mathrm e^{r_1 t}$. 
When we take the reproduction measure such that $R_{1}(dy)/y\in \mathcal{M}[0,1]$ we observe a relation to  branching processes in a random environment in the sense of \cite{AK}. To make this connection precise, let $\{(t_n,y_n)\}_{n\in\N}$ be the times at which an event occurs and their impacts. Then  $Y_n=\sum_{i=1}^{Y_{n-1}}X_i^{(n)}=Z_{t_n}$ is a branching process in random environment where for every $i\in\N$, $\P(X_i^{(n)}=2)=1-\P(X_i^{(n)}=1)=y_n$, and given $\{(t_n,y_n)\}_{n\in\N}$, the random variables $(X_i^{(n)})_{n\in\N}$ are independent.
 \end{example}

In order to study the processes rigorously, the following definitions are useful. Let 
\[ C_0=\{f \colon \N_0^V\to \mathbb{R} \colon \lim_{i\rightarrow \infty}f(z_i)=0, \forall (z_i)_{i\in \N}\text{ such that} \lim_{i\rightarrow \infty}|z_i|=\infty\}. \] 
We say that $g \in \widetilde C_0$ if $g=c+f$ for some $c\in \R$ and $f \in C_0$. We consider the process $(Z_t)_{t\geq 0}$ as a process in the one-point compactification $\bar\N_0^V$ of $\mathbb{N}_0^V$ by taking the minimal extension, that is to say that the generator at any function evaluated at the point at infinity is zero.

\begin{lemma}\label{lem:markov}
$(Z_t)_{t\geq 0}$ is a well-defined continuous time Markov chain with state space $\bar\N_0^V.$ Further, it is a conservative process, in the sense that for all $T>0$ 
\[ \P(Z_t\in \mathbb{N}_0^V, \forall t\in[0,T]\big|Z_0 \in \N_0^V)=1, \numberthis\label{conservativity} \] and the domain of its extended generator includes $\widetilde C_0$.
\end{lemma}
We recall that the domain of the extended generator is equal to the set of functions corresponding to the associated martingale problem, which will be spelt out in the proof of the lemma below. Note that if \eqref{conservativity} holds for all $T>0$, then this together with the continuity of measures implies
\[ \P\big(Z_t\in \mathbb{N}_0^V, \forall t\in[0,\infty)\big|Z_0 \in \N_0^V \big)=1, \]
i.e., almost surely, the process $(Z_t)_{t \geq 0}$ does not explode within finite time.
\begin{proof}[Proof of Lemma~\ref{lem:markov}]
The core of the proof is a domination argument and a control of the overall jump intensities (as required in \cite[Chapter 4, (11.9)]{EK}) in order to prevent explosion.

In order to verify that the process is conservative, we first show that the rate at which the process leaves any state $z\in \N_0^V$ is finite. Thus \eqref{eq:jumps_random} yields a $Q$-matrix, which generates the semigroup of a Markov chain on $\N_0^V.$ Observe that for all $u,v\in V, z_v\in \N_0$
\begin{align*}
\sum_{i=1}^{z_v}\int_0^1 \binom{z_v}{i}y^i(1-y)^{z_v-i}\frac{1}{y}M_{vu}(\dd y)=&\int_0^1 (1-(1-y)^{z_v})\frac{1}{y}M_{vu}(\dd y)\\
\leq &\int_0^1 z_v y\cdot \frac{1}{y}M_{vu}(\dd y)=z_v m_{vu}.
\end{align*}
Similar calculations hold for the reproduction and death. For the coalescence we get
\begin{align*}
\sum_{i=2}^{z_v}\int_0^1 \binom{z_v}{i}y^i(1-y)^{z_v-i}\frac{1}{y^2}\Lambda_{v}(\dd y)=&\int_0^1 (1-(1-y)^{z_v}-z_vy(1-y)^{z_v-1})\frac{1}{y^2}\Lambda_{v}(\dd y)\\
\leq &z_v(z_v-1)c_v.
\end{align*}
 Thus the total rate at which the process jumps out of state $z\in \N_0^V$ is bounded from above by
\[\sum_{v \in V} z_v \left[\sum_{u=1 }^K( m_{vu}+r_{vu}+d_{v}+(z_v-1)c_v)\right]<\infty.
\]
The fact that the process is conservative can be proved directly. However, we use a simple stochastic domination argument. Let $(Y_t)_{t \geq 0}$ be a one-dimensional process in our class with only one non-zero parameter which is $R=\sum_{v \in V}  \sum_{u \in V} R_{vu}$. Then one can couple the processes $(Y_t)_{t \geq 0}$ and $(Z_t)_{t \geq 0}$ in such a way that
\[ \P(|Z_t| \leq Y_t,~ \forall t \geq 0)=1. \]
As we saw in Example 1, $\mathbb{E}[Y_t]=\mathrm e^{rt}$ where $r=\sum_{v \in V} \sum_{u \in V} r_{vu}$.
The fact that $(Y_t)_{t \geq 0}$ is increasing together with the finiteness of its expectation imply that $\P(Y_t<\infty, \forall t\in [0,T])=1$ for all $T>0$. From the stochastic domination we conclude that $(Z_t)_{t \geq 0}$ is conservative.


Finally, let us study the extended generator of $(Z_t)_{t \geq 0}$ (and the associated martingale problem). We observe that for any  $(a^{(i)})_{i \in V}$, $(m^{(i)})_{i \in V}$ such that $m^{(i)}\geq 0$ and $
a^{(i)}\in \R$ for all $i \in V$,
$$
\Big(\sum_{i\in V} a^{(i)}\mathrm e^{-m^{(i)}Z^{(i)}_t}-\int_0^t A \big(\sum_{i\in V} a^{(i)} \mathrm e^{-m^{(i)}Z^{(i)}_s}\big)ds\Big)_{t \geq 0}
$$
is a martingale, where $A$ is the pointwise generator of $(Z_t)_{t \geq 0}$ that we describe in Equation \eqref{eq:generator}. The martingale property follows from \cite[Chapter 4, Problem 15 (page 263)]{EK}, since condition \cite[Chapter 4, (11.9)]{EK} is satisfied. It follows that all functions of the form $f(z)=\sum_{k=1}^r\sum_{i\in V}a^{(i)}_k\mathrm e^{-m^{(i,k)}Z^{(i)}_t}$, for $r\in \mathbb{N}$, $i \in V$, $m^{(i,k)}>0$ and $a^{(i)}_k\in \mathbb{R}$, are in the extended domain of $(Z_t)_{t \geq 0}$. Finally, as this family of functions is dense in $C_0$, an application of the Stone--Weierstrass theorem allows us to conclude that for every $f\in C_0$, writing
$$
M_t^f=f(Z_t)-\int_0^t A f(Z_s)ds,
$$
$(M^f_t)_{t \geq 0}$ is a martingale. Now, if $g=c+f$ for $c \in \R$, $M_t^g=g(Z_t)-\int_0^t A g(Z_s) \dd s=M_t^f+c$ and thus $(M_t^g)_{t \geq 0}$ is also a martingale. Hence, $g$ is in the domain of the extended generator of $(Z_t)_{t \geq 0}$.
\end{proof}

Besides the branching processes in random environment briefly discussed in~Example \ref{ex:bpre}, several members of this class of processes are known:

\begin{enumerate}
\item Choosing $\Lambda_v\in \mathcal{M}[0,1]$, $D_v=0$, $R_v=0$ and $M_{uv}=m_{uv}\delta_0$ leads to the block-counting process of the structured $\Lambda$-coalescent, see \cite{pitman, MS, LSt}.
\item\label{example-Felix}For $V=\{v\}$, $\Lambda_v= 0$, $D_v=\delta_p$, $R_v=r \delta_0$, $r>0$, we obtain a branching process with binomial disasters as discussed in \cite{HP}.
\item\label{example-simultaneousswitching} For $V=\{1,2\}$, $\Lambda_1=\delta_0$, $\Lambda_2=0$, $M_{12},M_{21} \in \mathcal M[0,1]$, and for all $v \in V$, $D_v=R_v=0$, we get the seed-bank coalescent with simultaneous migration \cite{BGKW18}.
\item\label{example-spatialseedbank} Fix $n \in \N$, $K>0$, $V=\{ v_i \colon i \in \{1,\ldots,n\} \} \cup \{ w_i \colon i \in \{1,\ldots,n\} \}$, $(e_i)_{i=1}^n \in \R^n$, $(d_i)_{i=1}^n \in \R^n$ and $(a(i,j))_{i,j \in \{1,\ldots,n\},i \neq j} \in \R^{n\times n}$. Then for $\Lambda_{v_i}= d_i \delta_0$, $\Lambda_{w_i}=0$, $M_{v_i v_j}= a(i,j) \delta_0$, $M_{v_i w_i}=e_i\delta_0$ and $M_{w_i v_i} = K e_i \delta_0$, $i,j=1,\ldots n$, we obtain the moment dual of a spatial seed bank model \cite[Model 1]{GdHO20}. Further, \cite[Model 2]{GdHO20}, the moment dual of the multi-layer seed-bank model also satisfies Definition~\ref{def:random_rate}, but with different migration measures. These examples are even included in Definition~\ref{def:fixed_rate}, whereas different choices of the measures $M_{v_i v_j}$ yield spatial variants of the seed-bank coalescent with simultaneous migration (fulfilling Definition~\ref{def:random_rate} only).
\item\label{example-subordinatorenv} For $V=\{v\}$, $\Lambda_v\in \mathcal{M}[0,1]$, $D_v=0$ and $R_v \in \mathcal{M}[0,1],$ we obtain coordinated branching coalescing processes \cite{GCSW} which arise as the moment dual of the Wright-Fisher model with selection in a (subordinator) random environment in the sense of \cite{BCM}. Indeed, this is the process $(Z_t)_{t \geq 0}$ in \cite[Theorem 3.2]{BCM} with the following choice of parameters: $b_1=\sigma_E=0$, $p(z,w)=z$, and the Lévy process $(Y_t)_{t \geq 0}$ being a decreasing Lévy process (with a nonpositive drift), i.e., the negative of a subordinator.
\item\label{example-hierarchical} For a general finite graph $V$, and for all $v,u \in V$ with $u \neq v$ and for some $R, D \in \mathcal M [0,1]$ and $c,c',c''>0$, the choice $R_v=R$, $D_v=D$ and $\Lambda_v=c\delta_0$, further, $M_{uv} = c'\delta_0+c''\delta_1$ yields the moment dual of the hierarchical Moran model introduced in \cite[Section 2.1]{Dawson}, in the case when there is no selection on the level of colonies. In the particular case $R=D=0$, this process is Kingman's coalescent with erosion \cite{FLS}.
\item $V=[K]^d \subset \mathbb Z^d$ where $[K]=\{ 1,\ldots,K\}$, and for all $v \in [K]^d$, $\Lambda_v=0$, $D_v=\xi_v^{-} \delta_0$, $R_v=\xi_v^+ \delta_0$ and $m_{vu}=\delta_0 \mathds 1_{ \{ |v-u| =1 \}}$, where $\{\xi_v^+\}_{v \in [K]^d}$ and $\{\xi_v^-\}_{v \in [K]^d}$ are two families of independent and identically distributed random variables in $[0,\infty)$, leads to a branching process whose expectation is a solution of the Parabolic Anderson Model (PAM), see \cite{KPAM}. In the context of the PAM, coordination and some consequences will be discussed in Sections~\ref{sec:PAM} and~\ref{sec-Var}, which will partially be extended to infinite graphs in Section~\ref{sec-infinity}. 
\end{enumerate}

The first example is classical, and so is the interpretation as a coordinated process: According to an underlying Poisson point process, coalescence events happen, and at each event, blocks decide independently according to $y\in[0,1]$ determined by $\Lambda$ whether or not to participate in the merger. 
Examples \ref{example-Felix}, \ref{example-simultaneousswitching}, \ref{example-spatialseedbank}, \ref{example-subordinatorenv} and \ref{example-hierarchical} are recent in the literature. In these cases, coordination (of migration respectively death and reproduction) was used to construct models that include interesting features. For all five models, moment duality results were proved. The PAM is a well-understood model with a large literature. Despite having a moment dual, it is not usually included in the class of models that can be studied using the techniques of population genetics.  We will introduce below \textit{the coordinated processes associated to the PAM}; all the members of this family have the same expectation, but radically different behaviour. 

Another classical example of a process of our class is the \emph{binary contact path process} \cite{Griffeath}, which is strongly related to the contact process. Since these processes are usually studied on infinite graphs, we will recall them in this context in Section~\ref{sec-infinity}.

\section{Moment Duality}\label{sec1}

Since the process $(Z_t)_{t\geq 0}$ is a pure jump Markov process with finite rates, it is straightforward to identify its generator, which we denote by $A.$ It acts on bounded measurable functions $f:\N_0^V\to\R$ by
\begin{equation}
Af( z)=\sum_{u,v\in V} A_{M_{vu}}f(z)+\sum_{u,v\in V} A_{R_{vu}}f( z)+\sum_{v\in V}A_{D_{v}}f(z)+\sum_{v\in V} A_{\Lambda_v}f(z),
 \end{equation}
where
\begin{align*}
A_{M_{vu}}f(z) & = \int_0^1\sum_{i=1}^{z_v}\binom{z_v}{i}[f(z- ie_v+ ie_u)-f( z)]y^i(1-y)^{z_v-i}\frac{1}{y}M_{vu}(\dd y), \\
A_{R_{vu}}f( z) & = \int_0^1 \sum_{i=1}^{z_v}\binom{z_v}{i}[f( z+ ie_u)-f( z)]y^i(1-y)^{z_v-i}\frac{1}{y}R_{vu}(\dd y), \\
A_{D_{v}}f(z) &= \int_0^1\sum_{i=1}^{z_v}\binom{z_v}{i}[f(z- ie_v)-f( z)]y^i(1-y)^{z_v-i}\frac{1}{y}D_v(\dd y), \\
A_{\Lambda_{v}}f(z) &= \int_0^1\sum_{j=2}^{z_v}\binom{z_v}{j}[f(z- (j-1)e_v)-f(z)](1-y)^jy^{z_v-j}\frac{1}{y^2}\Lambda_{v}(\dd y). \numberthis\label{eq:generator}
\end{align*}
Our goal is to derive a moment duality. As a first step, we derive a generator duality. Define for functions $f \in \mathcal C^2([0,1]^V,\R)$

\begin{equation}
Bf(x):=\sum_{u,v\in V} B_{M_{vu}}f(x)+\sum_{u,v\in V} B_{R_{vu}}f(x)+\sum_{v\in V}B_{D_{v}}f(x)+\sum_{v\in V} B_{\Lambda_v}f(x),
 \end{equation}
where
\begin{equation}\label{eq:dual-migration}
B_{M_{vu}}f(x):= \int_0^1[f(x+ e_vy(x_u-x_v))-f(x)]\frac{1}{y}M_{vu}(\dd y),
\end{equation}
\begin{equation}\label{eq:dual-branching}
B_{R_{vu}}f(x):= \int_0^1[f(x+e_vyx_v(x_u-1))-f(x)]\frac{1}{y}R_{vu}(\dd y),
\end{equation}
\begin{equation}\label{eq:dual-death}
B_{D_{v}}f(x):= \int_0^1[f(x+e_vy(1-x_v))-f(x)]\frac{1}{y}D_v(\dd y)
\end{equation}
and
\begin{equation}\label{eq:dual-coalescence}
B_{\Lambda_{v}}f(x):= \int_0^1[x_vf(x+e_vy(1-x_v))+(1-x_v)f(x-e_vyx_v)-f(x)]\frac{1}{y^2}\Lambda_{v}(\dd y).
\end{equation}
Formulas \eqref{eq:dual-migration}--\eqref{eq:dual-coalescence} are valid for measures that have no atoms at zero, which are extended to the case of measures having atoms at zero via a standard abuse of notation, which we will explain now. We first clarify how to understand these formulas if one of these measures equals $\delta_0$. If $M_{vu}=\delta_0$, then \eqref{eq:dual-migration} is to be understood as
\[ B_{M_{vu}} f(x)= (x_u-x_v) \frac{\partial f}{\partial x_v}(x).
\numberthis\label{eq:independent-migration} \]
If $R_{vu}=\delta_0$, then \eqref{eq:dual-branching} degenerates to 
\[ B_{R_{vu}} f(x)= x_v(x_u-1) \frac{\partial f}{\partial x_v}(x). \numberthis\label{eq:independent-branching}\]
If $D_v=\delta_0$, then \eqref{eq:dual-death} reads as
\[ B_{D_v} f(x)=(1-x_v) \frac{\partial f}{\partial x_v}(x). \numberthis\label{eq:independent-death} \]
Finally, if $\Lambda_v=\delta_0$, then we obtain the generator of the corresponding Wright--Fisher diffusion as the limit of \eqref{eq:dual-coalescence}:
\[ B_{\Lambda_v} f(x)=\frac{1}{2} x_v(1-x_v) \frac{\partial^2 f}{\partial x_v ^2 }(x). \numberthis\label{eq:WFdiffusion}\]
A general measure $\mu \in \mathcal M[0,1]$ can always be decomposed as $\mu=c\delta_0 + \mu'$ where $c\geq 0$ and $\mu'(\{0\})=0$. The general formulas for $B_{M_{vu}},B_{R_{vu}},B_{D_v}$ and $B_{\Lambda_v}$ are then obtained as a linear combination of the generators $B_{M_{vu}}$ in \eqref{eq:dual-migration} and \eqref{eq:independent-migration}, the ones $B_{R_{vu}}$ in \eqref{eq:dual-branching} and \eqref{eq:independent-branching}, the ones $B_{D_v}$ in \eqref{eq:dual-death} and \eqref{eq:independent-death}, respectively the ones $B_{\Lambda_v}$ in \eqref{eq:dual-coalescence} and \eqref{eq:WFdiffusion}. \\ \smallskip

Let $H:[0,1]^V\times \N_0^V \rightarrow [0,1]$ be such that for any $z \in  \N_0^V$ and $x \in[0,1]^V$
\[
H(x, z)=\prod_{v \in V} x_{v}^{z_v}
\]
with the convention that $0^0=1$. For $ z \in  \N_0^V$ and $ x \in[0,1]^V$ we write $H_{x}(z):=H(x, z)$ and $H_{z}(x):=H(x, z)$ in order to indicate the coordinate functions. Clearly, $H_x$ for $x\in[0,1]^V$ is in the domain of the generator $A$, and $H_z$ for $z \in \N_0^V$ is in the domain of the generator of $B$.

\begin{theorem}\label{thm:generator_duality}
For $H$ defined above, 
\[A H_x(z)=BH_z(x)\quad \forall z\in\N_0^V, x\in [0,1]^V.\]
\end{theorem}

\begin{proof} We check this for the four parts of the generator.
Migration:
\begin{eqnarray*}
A_{M_{vu}}H_{x}( z)&=&H_{x}( z)\int_0^1\sum_{i=0}^{z_v}\binom{z_v}{i}[x_v^{-i}x_u^i-1]y^i(1-y)^{z_v-i}\frac{M_{vu}(\dd y)}{y}\\
&=&H_{x}(z)x_v^{-z_v}\int_0^1\sum_{i=0}^{z_v}\binom{z_v}{i}[x_v^{z_v-i}x_u^i-x_v^{z_v}]y^i(1-y)^{z_v-i}\frac{M_{vu}(\dd y)}{y}\\
&=&H_{ x}(z)x_v^{-z_v}\int_0^1[(x_v(1-y)+x_uy)^{z_v}-x_v^{z_v}]\frac{M_{vu}(\dd y)}{y}\\
&=&\int_0^1[H_{z}(x+e_vy(x_u-x_v))-H_{z}(x)]\frac{M_{vu}(\dd y)}{y}\\
&=& B_{M_{vu}}H_{z}(x).
 \end{eqnarray*}
Reproduction:
\begin{eqnarray*}
A_{R_{vu}}H_{x}(z)&=&H_{x}( z)\int_0^1\sum_{i=0}^{z_v}\binom{z_v}{i}[x_u^i-1]y^i(1-y)^{z_v-i}\frac{R_{vu}(\dd y)}{y}\\
&=&H_{x}( z)\int_0^1[(x_uy+(1-y))^{z_v}-1]\frac{R_{vu}(\dd y)}{y}\\
&=&\int_0^1[H_{ z}(x+e_vyx_v(x_u-1))-H_{ z}( x)]\frac{R_{vu}(\dd y)}{y}\\
&=& B_{R_{vu}}H_{z}( x)
 \end{eqnarray*}
where in the third equality we used $x_v^{z_v}(x_uy+(1-y))^{z_v}=(x_v(1-y)+x_vx_uy)^{z_v}=(x_v-yx_v(1-x_u))^{z_v}$.\\
Death:
\begin{eqnarray*}
A_{D_{v}}H_{ x}( z)&=&H_{x}( z)\int_0^1\sum_{i=0}^{z_v}\binom{z_v}{i}[x_v^{-i}-1]y^i(1-y)^{z_v-i}\frac{D_{v}(\dd y)}{y}\\
&=&H_{x}( z)x_v^{-z_v}\int_0^1\sum_{i=0}^{z_v}\binom{z_v}{i}[x_v^{z_v-i}-x_v^{z_v}]y^i(1-y)^{z_v-i}\frac{D_{v}(\dd y)}{y}\\
&=&H_{x}(z)x_v^{-z_v}\int_0^1[(x_v(1-y)+y)^{z_v}-x_v^{z_v}]\frac{D_{v}(\dd y)}{y}\\
&=&\int_0^1[H_{z}(x+e_vy(1-x_v))-H_{z}(x)]\frac{D_{v}(\dd y)}{y}\\
&=& B_{D_{v}}H_{ z}( x).
 \end{eqnarray*}
 
Coalescence (this calculation in well-known, but we include it for completeness):
 
 \begin{eqnarray*}
A_{\Lambda_{v}}H_{ x}(z)&=&H_{ x}( z)x_v^{-z_v}\int_0^1\sum_{i=0}^{z_v-1}\binom{z_v}{i}[x_v^{i+1}-x_v^{z_v}](1-y)^iy^{z_v-i}\frac{\Lambda_{v}(\dd y)}{y^2}\\
&=&H_{ x}(z)x_v^{-z_v}\int_0^1[x_v(x_v(1-y)+y)^{z_v}+(1-x_v)x_v^{z_v}(1-y)^{z_v}-x_v^{z_v}]\frac{\Lambda_{v}(\dd y)}{y^2}\\\nonumber
&=&H_{ x}( z)x_v^{-z_v}\int_0^1[x_v(x_v+y(1-x_v))^{z_v}+(1-x_v)(x_v-yx_v)^{z_v}-x_v^{z_v}]\frac{\Lambda_{v}(\dd y)}{y^2}\\\nonumber
&=&\int_0^1[x_vH_{z}(x+e_vy(1-x_v))+(1-x_v)H_{ z}(x -e_vyx_v)-H_{ z}( x)]\frac{\Lambda_{v}(\dd y)}{y^2}\\\nonumber
&=& B_{\Lambda_{v}}H_{ z}( x).
 \end{eqnarray*}
 \end{proof}
We denote by $(X_t)_{t \geq 0}$ the Markov process on $[0,1]^V$ with infinitesimal generator $B$.
 
 \begin{corollary}\label{duality}
  $(X_t)_{t\geq 0}$ and the process $(Z_t)_{t\geq 0}$ of Definition \ref{def:random_rate} are moment duals, that is, for all $z\in\N_0^V, x\in[0,1]^V, t\geq 0$ we have
 \begin{equation}\label{eq:duality}
 \E_{x}[H(X_t,z)]= \E_{z}[H(x, Z_t)].
 \end{equation}
 \end{corollary}
 
 \begin{proof}
 This follows from Proposition 1.2 of \cite{JK} and Theorem \ref{thm:generator_duality}, since by our assumptions the rates are finite. Alternatively, using Theorem \ref{thm:generator_duality} and Lemma \ref{lem:markov} it is immediate to verify the conditions of Theorem 4.11 in \cite{EK}.
 \end{proof}
 
 The dual Markov process $(X_t)_{t\geq 0}$ can be explicitly represented as a $|V|$-dimensional jump diffusion.  Fix the parameters $\Lambda_v$, $D_v$, $R_{vu}$ and $M_{vu}$ in $\mathcal{M}[0,1].$ We define the following Poisson point processes (PPPs).\label{PPP's} For $v\in V$ let $N^{D_{v}}$ be a PPP on $(0,\infty)\times (0,1]$ with intensity measure $\frac{D_{v}(\dd y)}{y} \otimes \dd t $, and $N^{\Lambda_{v}}$ a PPP on $(0,\infty)\times [0,1] \times (0,1] $ with intensity measure $ \frac{\Lambda_{v}(\dd y)}{y^2} \otimes \dd t\otimes \dd \theta$. For $(v,u)\in V\times V$ let $N^{M_{vu}}$ be a PPP on $(0,\infty)\times (0,1]$ with intensity measure $\frac{M_{vu}(\dd y)}{y} \otimes \dd t $, and $N^{R_{vu}}$ a PPP on $(0,\infty)\times (0,1]$ with intensity measure $ \frac{R_{vu}(\dd y)}{y} \otimes \dd t$. Here, the notations $\dd t$, $\dd y$ and $\dd \theta$ stand for the Lebesgue measure on $[0,\infty)$ respectively $(0,1]$ and $[0,1]$. All PPPs involved are independent of each other and independent for different $v\in V$ respectively $(v,u)\in V\times V$. Let $(B_t)_{t\geq 0}=(B^{(v)}_t)_{t\geq 0, v\in V}$ be a $|V|$-dimensional standard Brownian motion independent of the PPPs. Then $(X_t)_{t\geq 0}=(X_t^{(v)})_{t\geq 0, v\in V}$ solves the system of SDEs 
 \begin{align}\label{eq:diffusion}
 \dd X_t^{(v)}=&\sum_{u\in V}\int_{y \in (0,1]} \big(y(X_{t-}^{(u)}-X_{t-}^{(v)})\big)  N^{M_{vu}}(\dd y,\dd t)\\ \nonumber
 &+\sum_{u\in V}  \int_{y \in (0,1]}\big( yX_{t-}^{(v)}(X_{t-}^{(u)}-1) \big) N^{R_{vu}}(\dd y,\dd t) \\ \nonumber
&+ \int_{y \in (0,1]} \big(y(1-X_{t-}^{(v)})\big)  N^{D_{v}}(\dd y,\dd t)\\ \nonumber
&+\int_{y \in (0,1]}\int_{\theta \in [0,1]} \big(y(\mathds 1_{\{\theta<X_{t-}^{(v)}\}}-X_t^{(v)})  \big)  N^{\Lambda_v}(\dd y, \dd \theta, \dd t)\\ \nonumber
&+\sum_{u\in V, u \neq v} (X_t^{(u)}-X_t^{(v)})  M_{vu}(\{0\})\dd t+\sum_{u\in V} X_t^{(v)}(X_t^{(u)}-1)  R_{vu}(\{0\})\dd t\\ \nonumber
&+(1-X_t^{(v)}) D_{v}(\{0\})\dd t+\sqrt{X_t^{(v)}(1-X_t^{(v)})}   \Lambda_v(\{0\})\dd B^{(v)}_t, \quad t\geq 0
 \end{align}
where $v\in V,$ with initial condition $X_0=(x^{(v)})_{v\in V}\in [0,1]^V.$ 

The above initial value problem gives a $|V|$-dimensional jump diffusions with non-Lipschitz coefficients. Existence and uniqueness results for such systems have recently drawn considerable interest, and we may refer to \cite{K, K1, BLP, xi2017jump} for existence and strong uniqueness results.

This dual process has an interpretation of the frequency process in the sense of population genetics, which is classical at least in the case without coordination. In that case, \eqref{eq:diffusion} reduces to 
\begin{eqnarray}\label{difclassic}
\dd X_t^{(v)}&=&\sum_{u\in V, u\neq v} (X_{t}^{(u)}-X_{t}^{(v)})  m_{vu}\dd t-\sum_{u\in V} X_{t}^{(v)}(1-X_{t}^{(u)})  r_{vu}\dd t\\& &
+(1-X_{t}^{(v)}) d_{v}\dd t+\sqrt{X_{t}^{(v)}(1-X_{t}^{(v)})}   c_v\dd B^{(v)}_t, \quad v\in V, t\geq 0.\nonumber
 \end{eqnarray}
The solution $(X_t)_{t\geq 0}$ of \eqref{difclassic} can then be understood as the frequency of one genetic type in a two-type population living in a structured environment of $|V|$ islands. More precisely, it is the stepping stone model with mutation and selection, see \cite{Kimura}. Denote the two types by $-$ or $+.$ Then \eqref{difclassic} describes the dynamics under the following assumptions:
 \begin{enumerate}
 \item $ m_{vu}$ is the rate at which individuals of island $v$ migrate to island $u$ (migration). 
 \item $ r_{vu}$ measures the selective disadvantage of type $-$ individuals situated on island $v$ against type $+$ individuals situated on island $u$ (selection). Note that the term $r_{vu}$ for $v \neq u$ is less classical than the one $r_{vv}$, nevertheless it receives an analogous interpretation,
 \item $ d_{v}$ is the rate at which individuals of type $+$ change into individuals of type $-$ on island $v$ (mutation from $+$ to $-$). 
 \item $ c_{v}$ measures the strength of the random genetic drift in island $v$. 
\end{enumerate}
In our dual process in \eqref{eq:diffusion}, the role of general measures, as opposed to Dirac measures at $0$, is compatible with this classical interpretation, which is now enriched by the possibility of large events that affect a positive fraction of the population.  Large migration events are considered for example in the seed-bank model with simultaneous switching, see \cite{BGKW18}.

\subsection{The nested coalescent and its dual}\label{sec:nested}
The nested coalescent is an object introduced recently \cite{BDLS}, which has already received some attention \cite{BRSS, D, LS}. Its purpose is to integrate speciation events and individual reproduction in the same model, in order to be able to trace ancestry at the level of species. Species can be regarded as islands (in the sense of a classical structured coalescent), meaning that individual ancestral lines inside each species coalesce according to some measure $\Lambda$, for example at pairwise rate one, just as in the Kingman coalescent. The difference is that species also perform a Kingman coalescent of their own, and when two species coalesce, the ancestral lines inside them are allowed to coalesce, again at pairwise rate one. Thus the nested coalescent consists of (independent) coalescents at individual level, nested inside an \lq external\rq\ coalescent at species level.  
In our framework, the block-counting process of the nested coalescent is given by choosing $\Lambda_v\in \mathcal{M}[0,1]$, $D_v=0$, $v\in V$, $R_{vu}=0$ and $M_{vu}=\delta_1, v\neq u$ in Definition \ref{def:random_rate}. The resulting process $(Z_t)_{t\geq 0}$ on $\N_0^V$ is given by the jumps

\begin{equation}
z \mapsto \left\{ \begin{array}{ll}
z+z_v(-e_v+ e_u), & \textrm{ at rate }1, \text{ for each } v,u\in V, v\neq u \\
z- e_v, & \textrm{ at rate } \Lambda_v(\{0\})\binom{z_v}{2}+\int_0^1\binom{z_v}{2}y(1-y)^{z_v-1}\frac{\Lambda_{v}(\dd y)}{y^2} \\
z- je_v, & \textrm{ at rate } \int_0^1\binom{z_v}{j+1}y^j(1-y)^{z_v-j}\frac{\Lambda_{v}(\dd y)}{y^2}, \qquad \textrm{ for }j \geq 2.
\end{array}
\right .
\end{equation}

If we impose $\Lambda_v=\Lambda$ for all $v\in V$ and ignore the empty islands, then up to labelling, $(Z_t)_{t\geq 0}$ is the block-counting process of a nested coalescent with individual $\Lambda$-coalescent and species Kingman coalescent. We now define the moment dual of the nested coalescent, which to our knowledge has not yet been introduced in the literature.

\begin{definition}[The nested Moran model]\label{def:nested_moran}
 Fix parameters $\Lambda_v\in \mathcal{M}[0,1]$, $D_v=0$, $R_{vu}=0$ and $M_{vu}=\delta_1$. Let $(X_t)_{t \geq 0}=(X_t^{(v)})_{t\geq 0, v\in V}$ be the solution of
 \begin{eqnarray}\label{diffusionmoran}
\dd X_t^{(v)}&=&\sum_{u \colon u\neq v} (X_{t-}^{(v)}-X_{t-}^{(u)}) \widetilde N^{M_{vu}}(\dd t)+\int_{y \in (0,1]} \int_{\theta \in [0,1]} y(\mathds 1_{\{\theta<X_{t-}^{(v)}\}}-X_{t-}^{(v)})    N^{\Lambda_v}(\dd y, \dd \theta,\dd t) \nonumber
\\&&+\sqrt{X_t^{(v)}(1-X_t^{(v)})}  \Lambda_v(\{0\}) \dd B^{(v)}_t, \quad t\geq 0\nonumber
 \end{eqnarray}
for $v\in V,$ where $N^{\Lambda_v}$ are independent Poisson point processes as in \eqref{eq:diffusion}, $\widetilde N^{M_{vu}}$ are independent Poisson point processes on $[0,\infty)$ with intensity $\dd t$ that are also independent of $\{ N^{\Lambda_v} \colon v \in V \}$, and $
(B_t^{(v)})_{t\geq 0, v\in V}$ is a standard $|V|$-dimensional Brownian motion independent of these Poisson point processes.  We call $(X_t)_{t\geq 0}$ the \emph{nested Moran model} with parameters $\Lambda_v, v\in V$.
\end{definition}

The relation between the nested Moran model and the classical Moran model becomes clear if one considers an initial condition $X_0=(X_0^{(v)})\in\{0,1\}^V$. Observe that in this case equation \eqref{diffusionmoran} reduces to
 \begin{eqnarray}
\dd X_t^{(v)}&=&\sum_{u: u\neq v} (X_{t-}^{(v)}-X_{t-}^{(u)})  \widetilde N^{M_{vu}} (\dd t).
 \end{eqnarray}
The connection becomes clear after observing that $\frac{1}{|V|}\sum_{v\in V} X_t^{(v)}$ is the frequency process of a Moran model with population size $|V|$. 

Further, considering Example~\ref{example-hierarchical} in Section~\ref{sec-modeldef}, we see that the nested Moran model is similar to the hierarchical Moran model in case all branching and death measures are zero. A substantial difference is that in the hierarchical case, there is also independent migration apart from the completely coordinated one, i.e., $M_{vu}=c'\delta_0+c''\delta_1$ for some positive $c',c''$.

The word \emph{nested} may seem slightly misleading in the context of Defintion \ref{def:nested_moran}, as the object we introduce is not a family of Moran models correlated by a Moran model, but rather a family of jump diffusions correlated by a Moran model. We use the name \emph{nested Moran model} in order to emphasize that it arises as the moment dual of the nested coalescent. This is the content of the next result, which is an immediate corollary of Theorem \ref{duality}.

\begin{corollary}[The nested coalescent and its dual]
 Fix parameters $\Lambda_v\in {\mathcal{M}}[0,1]$, $D_v=0$, $R_{vu}=0$ and $M_{vu}=\delta_1$. Then the block-counting process of the nested coalescent $(Z_t)_{t\geq 0}$ and the nested Moran model $(X_t)_{t\geq 0}$ with these parameters are moment duals, that is, for every $x \in [0,1]^V$, $z\in \N_0^V$ and $t>0$
\begin{equation}\label{duality_nested}
 \E_{x}[\prod_{v\in V} (X_t^{(v)})^{z_v}]= \E_{z}[\prod_{v\in V} (x_v)^{Z_t^{(v)}}].
 \end{equation}
\end{corollary} 

\begin{remark}
It seems plausible to generalize this construction to species  $\Lambda$-coalescents and even to more general nested coalescents (see \cite{D}) considering the Poisson processes governing the migration to be exchangeable instead of independent. 
\end{remark}


We now provide a criterion for the nested coalescent to come down from infinity, whose proof is based on the moment duality. Let us first recall the following notions related to coming down from infinity. We set $\bar z:=(z,z,...,z)$ and $\bar x:=(x,x,...,x)$, for some $z\in\N$ and $x\in [0,1]$, where both $\bar z$ and $\bar x$ have $|V|$ coordinates. Further, for $w \in \N_0^V$ we put $|w|:=\Vert w \Vert_1=\sum_{i=1}^n w_i$. 
\begin{definition}\label{def:comingdown}
We say that the structured branching coalescing process $(Z_t)_{t \geq 0}$ \emph{immediately comes down from infinity} if $ \lim_{m\rightarrow \infty} \lim_{z\rightarrow \infty} \p_{\bar z}(|Z_t|<m)=1$ for every $t>0$, it \emph{comes down from infinity} if $ \lim_{m\rightarrow \infty} \lim_{z\rightarrow \infty} \p_{\bar z}(|Z_t|<m)>0$ for every $t> 0$ and it \emph{does not come down from infinity} if $ \lim_{m\rightarrow \infty} \lim_{z\rightarrow \infty} \p_{\bar z}(|Z_t|<m)=0$ for every $t>0$.
\end{definition} Thanks to the strong Markov property, these three cases cover all possibilities, i.e., it cannot happen that $ \lim_{m\rightarrow \infty} \lim_{z\rightarrow \infty} \p_{\bar z}(|Z_t|<m)$ is zero for small $t$ but positive for large $t$, and it is also impossible that it is less than one for small $t$ but equal to one for large $t$. Further, thanks to the strong Markov property and the fact that the total number of particles of the nested coalescent is decreasing in $t$, coming down from infinity for this process implies that $\P(\limsup_{t \to \infty} |Z_t|<\infty)=1$. 

Let us now again consider the particular case when $(Z_t)_{t \geq 0}$ is the nested coalescent, let $(X_t)_{t\geq 0}$ be the nested Moran model, and define
\[\tau:=\inf\{t>0: X_t=(1,1,...,1)\}.\]

For any measurable set $A \subseteq [0,\infty]$ and for any $t>0$, we write $\p_{\infty}(|Z_t|\in A)=\lim_{z\rightarrow \infty}\p_{\bar z}(|Z_t|\in A)$. We have the following lemma. 
\begin{lemma}\label{lem:nocomingdown}
For all $t >0$, $\p_{\infty}(|Z_t|<\infty)= \sup_{x\in (0,1)}  \p_{\bar x}(\tau<t).$
\end{lemma} 
\begin{proof}
 Note that for $x \in (0,1)$,
$$
 \p_{\bar x}(\tau<t)=\p_{\bar x}(X_t=(1,...,1))=\lim_{z\rightarrow \infty}\E_{\bar x}[\prod_{v\in V} (X_t^{(v)})^{z}]=\lim_{z\rightarrow \infty}\E_{\bar z}[\prod_{v\in  V} x^{Z_t^{(v)}}]
$$
where we have used the duality in the last equation. This implies that
$$
\lim_{z\rightarrow \infty}\p_{\bar z}(|Z_t|<m)+x^m\geq  \p_{\bar x}(\tau<t)\geq x^m\lim_{z\rightarrow \infty}\p_{\bar z}(|Z_t|<m).
$$
Taking $x=1-m^{-1/2}$, one gets that
$$
\lim_{z\rightarrow \infty}\p_{\bar z}(|Z_t|<m)+\mathrm e^{-\sqrt{m}} \geq  \p_{\overline{ 1-m^{-1/2}}}(\tau<t).
$$
After observing that $\p_{\bar x}(\tau<t)$ is an increasing function of $x$, we conclude that 
$$
\liminf_{m\rightarrow \infty} \lim_{z\rightarrow \infty} \p_{\bar z}(|Z_t|<m)\geq \sup_{x\in (0,1)}  \p_{\bar x}(\tau<t).
$$
Now take $x=1-m^{-2}$ and observe that  
$$\p_{\overline{1-m^{-2}}}(\tau<t)\geq \mathrm e^{1/m}\lim_{z\rightarrow \infty}\p_{\bar z}(|Z_t|<m),$$
and taking $m$ to infinity this implies that
$$ \sup_{x\in (0,1)}  \p_{\bar x}(\tau<t)\geq \limsup_{m\rightarrow \infty}\lim_{z\rightarrow \infty}\p_{\bar z}(|Z_t|<m).$$
\end{proof}
Now the proof of the following corollary is trivial.
\begin{corollary}\label{cor:nocomingdown}
 The nested coalescent immediately comes down from infinity if,  for every $t>0$, $\sup_{x\in (0,1)} \p_{\bar x}(\tau<t)=1$, comes down from infinity if,  for every $t>0$, $\sup_{x\in (0,1)} \p_{\bar x}(\tau<t)>0$  and does not come down from infinity if,  for every $t>0$, $\sup_{x\in (0,1)}  \p_{\bar x}(\tau<t)=0$.
 \end{corollary}
 \begin{proof}
 The corollary follows directly from Lemma~\ref{lem:nocomingdown}. Indeed, by continuity of measures,
 \[ \P_{\infty}(|Z_t|<\infty)=\lim_{m\to\infty}\P_{\infty}(|Z_t|<m)=\lim_{m\to\infty}\lim_{z\to\infty} \P_{\bar z}(|Z_t|<m). \]
 \end{proof}
We note that in \cite{BDLS}, an equivalent condition for coming down from infinity for the nested coalescent was found, where the marginal coalescent of species can be an arbitrary $\Lambda$-coalescent instead of a Kingman one. Here, the only assumption on the measures corresponding to the $\Lambda$-coalescent of individual ancestral lines  (called `marginal gene coalescent' in that paper) and the one of species is that they have no mass at one, i.e., the probability that all species or all individuals of a given species merge simultaneously is zero. It was shown in \cite[Proposition 6.1]{BDLS} that such a nested coalescent comes down from infinity if and only if both the marginal gene coalescent and the marginal species coalescent both come down from infinity. Our Corollary~\ref{cor:nocomingdown} provides a simple alternative condition for coming down from infinity in the case when the marginal species coalescent is Kingman, using moment duality, including also the case when $\Lambda_v(\{ 1\})>0$. We expect that this result also extends to the case of a general $\Lambda$-coalescent for the species, but we refrain from presenting details.

This approach via duality may be applied to other models in the literature, like the Kingman coalescent with erosion \cite{FLS} that was introduced in Example~\ref{example-hierarchical} in Section~\ref{sec-modeldef}, but we defer such investigations to future work.

\subsection{Fixation of the advantageous trait in models with selection}

Being able to manipulate different mechanisms in the same mathematical framework allows us to translate known results about one mechanisms to find new results about a different mechanism, and also to come up with comparison arguments involving seemingly unrelated behaviours. 
First, we provide a simple equivalent condition on coming down from infinity for one-dimensional processes that exhibit no coalescence but death, using a comparison between these two effects. Thanks to moment duality, this has implications regarding fixation in the moment dual of the process.  Second, by comparing migration with death and using moment duality, we show that there are several examples of selection that lead to almost sure fixation in a structured population. The latter result is new and seems to be interesting from a biological perspective. 


\begin{example}[Coming down from infinity without coalescence]
In \cite{S}, a necessary and sufficient condition for coming down from infinity for $\Lambda$-coalescents was provided. For structured processes with coalescence, the results of \cite{BDLS,BGKW18} show that different coalescance mechanisms at different vertices and coordinated migration can change the behaviour of the process with this respect radically, see also Corollary~\ref{cor:nocomingdown} in the present paper. It is natural to ask whether coming down from infinity is possible for structured branching coalescing processes exhibiting no coalescence but only death, migration and reproduction. While we expect that the answer to this question is still nontrivial due to a competition between death and migration if the set $V$ has at least two elements, for $V=\{ v \}$ the answer is straightforward and easy to verify, based on a comparison between death and coalescence.

Let us recall the notion of coming down from infinity, coming down from infinity immediately, and not coming down from infinity from Definition~\ref{def:comingdown}. We are interested in the case $V=\{ v \}$, where we ignore the index $v$ in the nomenclature of the corresponding measures and the process, writing simply $D,R,\Lambda$ and $ (Z_t)_{t \geq 0}$. 

\begin{proposition}\label{prop:nocomingdownD}
Assume that $|V|=1$ and $\Lambda=0$. Then the structured branching coalescing process $(Z_t)_{t \geq 0}$ comes down from infinity if and only if $D$ has an atom at one.
\end{proposition}
Before proving this proposition, let us explain some of its consequences. In case $D=\delta_1$, starting from an infinite number of particles, the process is extinguished (i.e., absorbed at zero) after an exponentially distributed time. Thanks to Proposition~\ref{prop:nocomingdownD} and the linearity of rates in \eqref{eq:jumps_random} with respect to $D=D_v$, it follows that the same holds whenever $D$ has an atom at one, in particular coming down from infinity never happens immediately if $\Lambda=0$. Let us note that coming down from infinity up to time $t$ with positive probability is equivalent to reaching 1 with positive probability up to time $t$ for the moment dual $(X_t)_{t \geq 0}$ starting from any initial condition in $(0,1)$. Indeed, for $x \in (0,1)$ and $t>0$ we have
\[ \E_{\infty} [x^{Z_t}]=\lim_{n\to\infty} \E_{n} [x^{Z_t}] = \lim_{n\to\infty} \E_{x}[X_t^{n}] = \P_x(X_t=1), \numberthis\label{CDFI}\]
which is positive if and only if the process comes down from infinity. In case there is selection in the model (i.e., $R \neq 0$), reaching 1 can be interpreted as fixation of the advantageous trait.
\begin{proof}[Proof of Proposition~\ref{prop:nocomingdownD}]
We will study three cases, when there is mass at one, when the mass is concentrated at zero and when there is no mass at either one nor zero. We will combine these cases in order to obtain the whole spectrum.

First, it is clear that if $D$ has an atom at 1, then the process $(Z_t)_{t \geq 0}$
gets extinguished after an exponentially distributed time, and hence in particular it comes down from infinity (further, if $D$ is a multiple of $\delta_1$, then it stays infinite until the extinction). Thus, the condition of the proposition is sufficient for coming down from infinity. 

For the rest of the proof, let us assume that $D$ has no atom at 1. Our goal is to show that the process does not come down from infinity. Since reproduction can only increase the value of the process, we assume without loss of generality that $R=0$. 

Now, let us first consider the case when $D=\delta_0$. Our process $(Z_t)_{t \geq 0}$ is a pure death chain with jumps $n \to n-1$ at rate $n$. Hence, by \eqref{eq:diffusion}, the dual process $(X_t)_{t \geq 0}$ is deterministic, it equals the unique solution $(x(t))_{t \geq 0}$ of the ODE
\[ \frac{\dd}{\dd t} x(t)=1-x(t) \]
with $x(t)=(1-(1-x(0))\e^{-t})$, $t \geq 0$.
We observe that $\E_{n} [(x(0))^{Z_t}]=\E_x[X_t^n]=(x(0)\e^{-t}+(1-\e^{-t}))^n$, and thus $Z_t$ is a binomial random variable with parameter $n$ and $\e^{-t}$. Either from this or from Equation \eqref{CDFI}, we conclude that the process $(Z_t)_{t \geq 0}$ does not come down from infinity.

Second, if $D$ has no atom at zero, then we define a coalescence measure $\widehat \Lambda$ according to
\[ \frac{\widehat \Lambda(\dd y)}{y^2} = \frac{D(\dd y)}{y}, \qquad y \in (0,1]. \numberthis\label{DLambdacoupling} \]
Then, since $D$ is a finite measure, we have $\int_{(0,1]} \frac{\widehat \Lambda(\dd y)}{y}<\infty$. Consequently, by \cite[Theorem 8]{pitman}, the $\widehat \Lambda$-coalescent does not come down from infinity. 
According to the proof of \cite[Corollary 2]{S}, it is even true that if we define a process $(Y_t)_{t \geq 0}$ similarly to the block-counting chain of the $\widehat \Lambda$-coalescent but having downward jumps of size $k$ instead of $k-1$ at rate $\int_0^1 x^{k-2} (1-x)^{b-k} \widehat \Lambda(\dd x) $, given that there are $b \geq k$ blocks, then $(Y_t)_{t \geq 0}$ does not come down from infinity. Now, according to \eqref{eq:jumps_random} and \eqref{DLambdacoupling}, starting from $b \in \N$ blocks at the moment, the rate at which $(Y_t)_{t \geq 0}$ jumps to $b-k$, $k=2,\ldots,b$, is 
\[ \binom{b}{k} \int_0^1 x^{k-2} (1-x)^{b-k} \widehat \Lambda(\dd x) = \binom{b}{k}  \int_0^1 x^{k-1} (1-x)^{b-k} D(\dd x). \]
Now, thanks to \eqref{eq:jumps_random}, started from $b \in \N$, our process $(Z_t)_{t \geq 0}$ has the same jump rates as $(Y_t)_{t \geq 0}$, plus additionally downward jumps of size 1 at rate
\[ b \int_0^1 (1-x)^{b-1} D(\dd x). \numberthis\label{singledeath} \]
This additional rate however depends linearly on $b$. Hence, using the same arguments as for $D=\delta_0$, one can easily verify that $(Z_t)_{t \geq 0}$ does not come down from infinity. The case of a general $D$ (without an atom at 1) follows from the linearity of the transition rates in the second line of \eqref{eq:jumps_random} with respect to the death measure $D=D_v$.
\end{proof}
\end{example}
\begin{remark}
As we have seen, given $D$, the construction \eqref{DLambdacoupling} results in a coalescence measure $\widehat \Lambda$ satisfying $\int_{(0,1]} \frac{\widehat \Lambda(\dd y)}{y}<\infty$. \cite[Theorem 8]{pitman} implies that the associated $\widehat \Lambda$-coalescent has dust, i.e.~if it is started from an infinite number of blocks, then there is a positive proportion of singletons for any $t>0$. This is a stronger condition than not coming down from infinity. E.g., the Bolthausen--Sznitman coalescent $\widehat \Lambda(\dd y)=\dd y$ has no dust, but it does not come down from infinity (cf.~\cite{S}).

Conversely, one could think of defining a death measure $D$ according to \eqref{DLambdacoupling} given a coalescence measure $\widehat \Lambda$ having no atom at zero. Then, for a process satisfying Definition~\ref{def:random_rate} and having zero death measure and coalescence measure $\widehat\Lambda$, we can consider the process where instead of coalescence according to $\widehat \Lambda$ there is death according to $D$, whereas reproduction is unchanged. This process, if it is well-defined, dominates the process with coalescence stochastically from below. However, the measure $D$ is only finite if $\int_{(0,1]} \frac{\widehat \Lambda(\dd y)}{y}<\infty$. Otherwise, the death rate \eqref{singledeath} of single individuals equals $+\infty$, and one can intuitively say that the process with death instead of coalescence is the degenerate process jumping to zero immediately after time $t=0$, whatever the initial condition is.
\end{remark}

\begin{example}\label{example-peripatric}
Now, consider the peripatric coalescent $(Z_t^{(1)},Z_t^{(2)})_{t\geq 0}$ defined analogously to \cite{LM}, but with corordinated migration (instead of independent one). This is one of the processes satisfying Definition~\ref{def:random_rate}, with $V=\{ 1, 2 \}$, where for some $c>0$, we have $\Lambda_1=c \delta_0$, $\Lambda_2=0$, $M_{21},M_{12} \in \mathcal M(0,1]$, $R_{11}=\alpha'\delta_0$ and $R_{22}=\alpha\delta_0$ for some $\alpha',\alpha\geq 0$, and all other measures are equal to zero. The moment dual of the arising process can be interpreted as follows: the vertex 2 is a continent with a large population and the vertex 1 is an island with a smaller one, there is migration in both directions and selection at both locations, but random genetic drift plays a role only on the island. 

We are interested in sufficient conditions under which $Z_t^{(2)}$ tends to infinity almost surely, which is equivalent to almost sure fixation of the fitter type in the dual process. Such an assertion follows as soon as we can verify that $Z_t \to \infty$ almost surely as $t \to\infty$ for a process $(Z_t)_{t \geq 0}$ satisfying
\[ Z_t^{(2)} \geq Z_t, \qquad \forall t \geq 0 \numberthis\label{peripatricdomination} \]
realizationwise, given that $Z_0^{(2)}=Z_0$. In order to construct a process that dominates $(Z_t^{(2)})_{t \geq 0}$ from below in this sense, we can remove migration from vertex 1 to vertex 2 from the model. Then, the population on vertex 1 does not influence the one on vertex 2, and hence we can ignore the population on vertex 1 and consider migration from vertex 2 to 1 as death. This gives rise to the one-dimensional process $(Z_t)_{t \geq 0}$ on vertex set $\{ 2 \}$ with the following measures: $\Lambda_2=0$ for the coalescence, $R_{22}=\alpha \delta_0$ for the reproduction (as for $(Z_t^{(2)})_{t \geq 0}$), $M_{22}=0$ for the migration and $D_2=M_{21}$ for the death. It is easy to see $(Z_t)_{t \geq 0}$ satisfies \eqref{peripatricdomination} realizationwise.

For a general migration measure $M_{21} \in \mathcal M[0,1]$, proving that $Z_t \to \infty$ almost surely as $t \to\infty$ may be involved. However, such results are available for Dirac measures. Note that for $D_2=p \delta_p$, $p \in (0,1)$, according to \eqref{eq:dual-death}, the part of the generator of the dual corresponding to death reads as
\[ B_{D_2} f(x)= \int_0^1 [f(x+y(1-x))-f(x)] \frac{1}{y} D_v(\dd y) = [f(x+p(1-x))-f(x)], \]
where we wrote $x$ instead of $x_2$ everywhere for simplicity. Now note that if a Markov process with values in $[0,1]$ jumps from $x \in [0,1]$ to $x+p(1-x)$ at rate $p$, then one minus this process jumps from $1-x$ to $(1-p)(1-x)$ at the same rate. This together with \eqref{eq:independent-branching} yields that if $(N_t)_{t \geq 0}$ denotes the dual of $(Z_t)_{t \geq 0}$, then $(1-N_t)_{t \geq 0}$ has generator
\[ \mathcal L f(x)= \alpha x (1-x) f'(x) + f((1-p)x)-f(x). \]
This relates our setting to the one of \cite{HP}, where processes with generators of the form
\[ \mathcal L f (x) = \alpha (x) f'(x) + f((1-p)x)-f(x) \numberthis\label{pgenerator} \]
were studied, where the domain of the generator $\mathcal L$ consists of continuously differentiable functions $f \colon [0,\nu) \to \R$ for $\nu>0$ fixed, and $\alpha \colon [0,\nu) \to \R$ differentiable. It was showed in \cite[Theorem 1]{HP} that if we have $\sup_{x \in (0,\nu)} \frac{\alpha(x)}{x} < -\log (1-p)$, then the process tends to zero almost surely. Now, if we choose $\nu=1$ and $\alpha(x)=\alpha x(1-x)$, $\alpha>0$, then this condition is equivalent to 
\[ \alpha < -\log (1-p), \numberthis\label{alphacond} \]
which can always be guaranteed by a suitable choice of $\alpha>0$ as long as $p \in (0,1)$. Thus, if $\alpha$ and $p$ satisfy \eqref{alphacond}, then it follows by \cite[Theorem 1]{HP} that $N_t \to 1$ almost surely as $t \to \infty$. This implies almost sure fixation of the fitter type in the peripatric coalescent with selection and coordinated migration thanks to \eqref{peripatricdomination}.

This can be generalized to the case when $D_2=M_{21}$ is compactly supported within $(0,1]$, i.e., there exist $p^->0$ such that $\Lambda([0,p^-)) = 0$. In this case, we define $\bar c= \int_{[0,1]} \frac{D_2(\dd y)}{y}$; this number is finite by assumption. Then the process $(N_t)_{t \geq 0}$ is stochastically dominated from below by the moment dual of the process having death measure $D_2^- = \bar c p^- \delta_{p^-}$ and all the other measures equal to the ones corresponding to $(Z_t)_{t \geq 0}$. Hence, if $\alpha$ is such that \eqref{alphacond} holds with $p=p^-$, then $N_t \to 1$ almost surely as $t \to\infty$, which again implies almost sure fixation.
\end{example}

\section{Coordination and expectation}\label{sec-E}
In this section we show that for the process $(Z_t)_{t \geq 0}$ introduced in Definition~\ref{def:random_rate}, for $v \in V$, the expectation process $t \mapsto \E[Z_t^{(v)}]$ equals the unique solution of a linear differential equation depending only on the total mass of the underlying measures $M_{uw}$, $D_{w}$ and $R_{uw}$, $u,w \in V$. This is true under the assumption that there is no coalescence (i.e., $\Lambda_v=0$ for all $v \in V$), but we will also provide some extensions to the case of nonzero coalescence.

\begin{lemma}\label{lem:expectationequality}
Let the collections of measures $(D_v)_{v \in V}$, $(R_{vu})_{v,u \in V}$, $(M_{vu})_{v,u \in V}$ and $(\Lambda_v)_{v \in V}$ satisfy Definition~\ref{def:random_rate} with $\Lambda_v=0$ for all $v \in V$. Define $(f(t,v))_{t \in [0,\infty), v \in V}=(\E[Z_t^{(v)}])_{t\in [0,\infty), v \in V}$. Then $(f(t,v))_{t \in [0,\infty), v \in V}$ is the unique solution of

\begin{equation}\label{magic}
\frac{\mathrm d}{\mathrm d t} f(t,v) = \sum_{u \in V}( f(t,u) m_{uv}-f(t,v) m_{vu}) - f(t,v) d_{v} + \sum_{u \in V} f(t,u) r_{uv}, \qquad v \in V, t\in [0,\infty)
\end{equation}
with initial condition $f(0,v)=z_0^{(v)}=Z_0^{(v)} \in \R^d$, $v \in V$.
\end{lemma}
\begin{proof}
The ODE in \eqref{magic} is linear with continuous coefficients and thus has a unique solution.
We fix $v\in V$ for the proof. From the proof of Lemma \ref{lem:markov} we know that $\E[|Z_t|]<\infty$ for all $t>0$ and from Lemma \ref{lem:markov} that $f^k_v(z)=\min \{z_v,k \}$ is in the extended domain for all $k \in \N$ and $v \in V$. From monotone convergence we conclude that $f_v(z)=z_v$ is also in the extended domain, for all $v \in V$. Applying the form \eqref{eq:generator} of the generator for $f_v(z)=z_v$, Dynkin's formula together with the linearity of expectation implies
\[ \begin{aligned}
\E&\big[Z_t^{(v)}\big]-Z_0^{(v)} = \E \Big[ \int_0^t \Big( \sum_{u \in V} \Big[ \int_0^1 \,  \sum_{i=1}^{Z_s^{(u)}} \binom{Z^{(u)}_s}{i} iy^i(1-y)^{Z_s^{(u)}-i} \frac{1}{y} M_{uv}(\dd y) \\ & \quad  - \int_0^1 \sum_{i=1}^{Z_s^{(v)}} \binom{Z^{(v)}_s}{i}i y^i(1-y)^{Z_s^{(v)}-i} \frac{1}{y} M_{vu}(\dd y) \Big]  \\ & \quad - \int_0^1 \sum_{i=1}^{Z^{(v)}_s}  \binom{Z^{(u)}_s}{i} i y^i(1-y)^{Z_s^{(u)}-i} \frac{1}{y}D_{v}(\dd y) 
 + \sum_{u \in V} \int_0^1 \sum_{i=1}^{Z_s^{(u)}} \binom{Z_s^{(u)}}{i} i y^i (1-y)^{Z_s^{(u)}-i} \frac{1}{y} R_{uv}(\dd y) \Big) \dd s \Big] 
\\
& = \E \Big[ \int_0^t  \Big( \sum_{u \in V} \Big[ \int_0^1 \, \E\big[\mathrm{Bin}(Z_s^{(u)},y) | Z_s\big] \frac{1}{y}M_{uv}(\dd y)-\int_0^1 \, \E\big[\mathrm{Bin}(Z_s^{(v)},y) | Z_s\big] \frac{1}{y}M_{vu}(\dd y) \Big] \\ & \quad -  \int_0^1  \, \E\big[\mathrm{Bin}(Z^{(v)}_s,y) | Z_s\big] \frac{1}{y}D_{v}(\dd y)
+ \sum_{u \in V} \int_0^1 \, \E\big[\mathrm{Bin}(Z^{(u)}_s,y) | Z_s\big] \frac{1}{y}R_{uv}(\dd y)\Big) \dd s \Big],
\end{aligned} \numberthis\label{eq:expectations}
\]
where $\mathrm{Bin}(n,p)$ denotes a binomially distributed random variable with parameters $n \in \N$ and $p \in [0,1]$. 
Let us show that for example the term for fixed $u \in V$ corresponding to migration between $u$ to $v$ depends only on the total mass of $M_{uv}$ and $M_{vu}$:
\[ \int_0^1 \, \E\big[\mathrm{Bin}(Z_s^{(v)},y)|Z_s\big] \frac{1}{y}M_{uv}(\dd y)-\int_0^1 \, \E[\mathrm{Bin}(Z_s^{(u)},y)|Z_s] \frac{1}{y}M_{vu}(\dd y) = Z^{(u)}_s m_{uv}-Z^{(v)}_s m_{vu}. \] 
Analogously, we obtain for the death term
\[ -  \int_0^1  \, \E\big[\mathrm{Bin}(Z^{(v)}_s,y) | Z_s\big] \frac{1}{y}D_{v}(\dd y) = -Z_s^{(v)} d_{v}\]
and for the reproduction term for fixed $u \in V$ with offspring in $v \in V$
\[ \int_0^1 \, \E\big[\mathrm{Bin}(Z^{(u)}_s,y) | Z_s\big] \frac{1}{y}r_{uv}(\dd y) = Z_s^{(u)} r_{uv}. \]
Since our process $(Z_t)_{t \geq 0}$ is nonnegative and $\E[\int_0^t Z_s^{(u)} m_{uv} \dd s]<\infty$ for all $u \in V$, the Fubini--Tonelli theorem implies that we can interchange the outermost expectation (with respect to $Z_s^{(v)}$) with the integration from 0 to $t$, and similarly for the death and reproduction terms. Performing this and differentiating with respect to $t$, we obtain that $(f(t,w))_{t \geq 0, w \in V}=(\E[Z_t^{(w)}])_{t \geq 0, w \in V}$ is a solution to \eqref{magic}.
\end{proof}
Since the system of linear ODEs has an unique solution, the previous lemma implies that the expectation is invariant under coordination of migration, birth and death, as long as the coalescence measures are zero and the total masses of the other measures are unchanged.

The previous lemma works for coordinating events that affect single individuals. As calculating the expectation for the fully coordinated process is in general simpler, this provides a general machinery to calculate expectations, see the following example and (in the context of the PAM) Example~\ref{ex:intermediatePAM}.
\begin{example}
Consider a one-dimensional pure death process $(Z_t)_{t \geq 0}$ with $D=d\delta_0$ (where $D$ and $R$ denote the single death and reproduction measures, respectively). Our approach to calculate its expectation is to consider the fully coordinated process $(\bar Z_t)_{t \geq 0}$ with $D=d\delta_1$, where all the individuals die simultaneously at a random time $\mathcal T$ which is exponentially distributed with parameter $d$. Then for $t \geq 0$, $\E_n[Z_t]=\E_n[\bar Z_t]=n\P(\mathcal T>t)=ne^{-dt}$. 

The same principle can be used to calculate the expectation of a Yule process $(Z_t)_{t \geq 0}$ with branching parameter $r>0$, i.e., $R=r\delta_0$ and all other parameters equal to zero. As we saw in Example 1, in this case the fully coordinated process ($R=r\delta_1$) admits the representation  $\bar Z_t=n 2^{W_{rt}}$ where $(W_t)_{t \geq 0}$ is a standard Poisson process. This implies that $\E_n[\bar Z_t]=n\E[2^{W_{rt}}]$. Now, using that the probability generating function of a Poisson random variable with rate parameter $rt$ evaluated at $x$ is $\E[x^{W_{rt}}]=e^{rt(x-1)}$ we conclude that $\E_n[Z_t]=\E_n[\bar Z_t]=ne^{rt}$. 

If we now change the notation and consider a process $(Z_t)_{t \geq 0}$ with $D=d\delta_0$, $R=r\delta_0$ and all other parameters equal to zero, we can combine the previous examples to observe that the fully coordinated process $(\bar Z_t)_{t \geq 0}$ defined via $\bar Z_t=n 2^{W_{rt}}\mathds 1_{\{\mathcal  T>t\}}$ satisfies $\E_n[Z_t]=\E_n[\bar Z_t]=n\E[2^{W_{rt}}]\P(\mathcal  T>t)=ne^{(r-d)t}$. It is interesting to see that the fully coordinated process and the birth-death branching process have the same expectation at any deterministic time $t\geq 0,$ but very different path behaviour. The first one will be extinguished almost surely at time $\mathcal T$ regardless of the reproduction rate, while if $r-d>0$ the birth-death branching process tends to infinity with positive probability.
\end{example}

The case $\Lambda_v\neq 0$ does not allow such a clean representation. It is clear that no such result can hold in general, i.e., coordination has an effect on the expectation in the presence of pairwise interaction. To see this, think of the expectation of the block-counting process of a Kingman coalescent (no coordination). It is known that at $t>0$ started from infinity the expectation is finite (see \cite{BB} and the references therein), while the expectation for a star-shaped coalescent (full coordination, \cite{pitman}) started at infinity is always infinite. It is not hard to see that starting both processes with 3 blocks will already lead to processes with different expectations.   However, it is still possible to use the idea of Lemma \ref{lem:expectationequality} to some extent. We state a result for the Kingman case $\Lambda_v=\delta_0$. 

\begin{proposition}\label{lem:expectationequality2}
Let the collections of measures $(D_v)_{v \in V}$, $(R_{vu})_{v,u \in V}$ and $(M_{vu})_{v,u \in V}$ and $(\Lambda_v)_{v \in V}$ satisfy Definition~\ref{def:random_rate} with $\Lambda_v=c_v\delta_0$. Let $(f(t,v))_{t \geq 0, v \in V}$ be any solution of
\begin{eqnarray}\label{magic2}
\frac{\dd}{\dd t} f(t,v) &=& \sum_{u \in V}( f(t,u) m_{uv}-f(t,v) m_{vu}) - f(t,v) d_{v} + \sum_{u \in V} f(t,u) r_{uv}\nonumber\\&&-(f(t,v)^2-f(t,v))\frac{c_v}{2}, \qquad v \in V, t >0,
\end{eqnarray}
with $f(0,v)=z_0^{(v)} \in \R$, $v \in V$.
Then $\E[Z_t^{(v)}]\leq f(t,v)$.
\end{proposition}
\begin{proof}
The proof follows from the proof of Lemma \ref{lem:expectationequality} together with Jensen's inequality. Indeed,
$$
-\E \Big[ \binom{Z_s^{(v)}}{2} \Big]=-\frac{1}{2}\E \Big[( Z_s^{(v)})^2- Z_s^{(v)}\Big]\leq -\frac{1}{2}\E [ Z_s^{(v)}]^2+\frac{1}{2}\E [ Z_s^{(v)}],
$$
which deals with the additional term.
\end{proof}

 \subsection{Variations of the PAM branching process}\label{sec:PAM}
 The Parabolic Anderson Model is a classical mathematical object that has attracted a lot of attention in the recent decades. In its context, we work on the subgraph of $\mathbb Z^d$ spanned by the vertex set $V=[K]^d$, where $[K]=\{ 1,\ldots, K \}$ for $K \in \N$. Consider the Cauchy problem for the heat equation with random coefficients and localized initial datum: Let $\bar\xi^+=\{\xi_v^+\}_{v\in [K]^d}$ and $\bar\xi^-=\{\xi_v^-\}_{v\in [K]^d}$ be two families of independent and identically distributed random variables with values in $\R^+$, where we write $[K]=\{ 1,\ldots, K \}$ for $K \in \N$. For $v \in [K]^d$ let us define $\xi_v := \xi^+_v-\xi^-_v$, and let us put $\bar \xi=\{ \xi_v \}_{v \in [K]^d}$. Then the Cauchy problem is
 \begin{eqnarray}
\frac{ \dd}{\dd t} f(t,v)&=&\sum_{u \colon |u-v|=1} [f(t,u)-f(t,v)]+\xi_vf(t,v) \label{PAM}\\
 f(0,v)&=&\mathds 1_{\{v=\bar 0\}}.\nonumber
 \end{eqnarray}
It is well-known that conditionally on $(\bar\xi^+,\bar \xi^-)$ the branching process $(Z_t)_{t \geq 0}$ which goes from the state $\bar z$ to the state $\bar z-e_v+e_u$ at rate $z_v$ if $u$, $v$ are neighbouring vertices and to the state $\bar z+e_v$ at rate $z_v \xi^+_v$ and $\bar z-e_v$ at rate $z_v \xi^-_v$, where $\xi^+_v-\xi^-_v=\xi_v$,  has the property that under mild conditions on $\bar\xi=\{\xi_v\}_{v\in [K]^d}$,   $f(t,v)= \E[Z_t^{(v)}]$ is a solution to the PAM \cite{GM}. (Note that conditional on $\bar \xi$, this solution is unique according to Lemma~\ref{lem:expectationequality}.) For this reason $(Z_t)_{t \geq 0}$ is studied in \cite{OR1,OR2,OR3}. We note that the results of the present section remain valid if we replace $[K]^d$ with a discrete torus, i.e., if the notion of neighbouring vertices is taken with respect to periodic boundary conditions.
 
As the process $(Z_t)_{t \geq 0}$ is a branching process, one can use moment duality, for example, to estimate the probability that there is at least one individual (in the branching process) in a certain position, using an ODE. Indeed, imagine that the branching process starts with one individual in the island $\bar 0$ and we are interested in knowing if at a fixed time $t>0$ there is some individual in position $v$. Taking $\bar x= (1,1,...,1)-e_v(1-\varepsilon)$, $\bar z= e_v$, 
\begin{equation}\label{extinction}
\mathbb{P}_{e_v}(Z_t^{(v)}=0)= \lim_{\varepsilon\rightarrow 0} \E_{e_v}[\varepsilon^{Z_t^{(v)}}]= \lim_{\varepsilon\rightarrow 0}\E_{\bar x}[X_t^{(v)}] 
 \end{equation}
and 
\begin{equation}\label{occupancy}
\mathbb{P}_{e_u}(Z_t^{(v)}=0)=\lim_{\varepsilon\rightarrow 0}  \E_{e_u}[\varepsilon^{Z_t^{(v)}}]= \E_{(1,1,...,1)-e_v}[X_t^{(u)}].
 \end{equation}
This approach seems not to be explored yet in the PAM literature. Indeed, the behaviour of this branching process is a classical problem that has been solved to a great extent only recently \cite{OR1,OR2,OR3} using different techniques.  Note that in the case without death, $(X_t^{(v)})_{v\in V,t\geq 0}$ is the solution of a system of ordinary differential equations
 \begin{eqnarray}
\frac{ \dd}{\dd t}X_t^{(v)}&=&\sum_{u \colon |u-v|=1} [X_t^{(u)}-X_t^{(v)}]dt+\xi_v X_t^{(v)}(1-X_t^{(v)})dt \label{PAMdual}
 \end{eqnarray}
which is easy to solve numerically and seems plausible to study mathematically (see Figure \ref{imagen}).

\begin{figure}\label{imagen} \centering
\includegraphics[scale=.467]{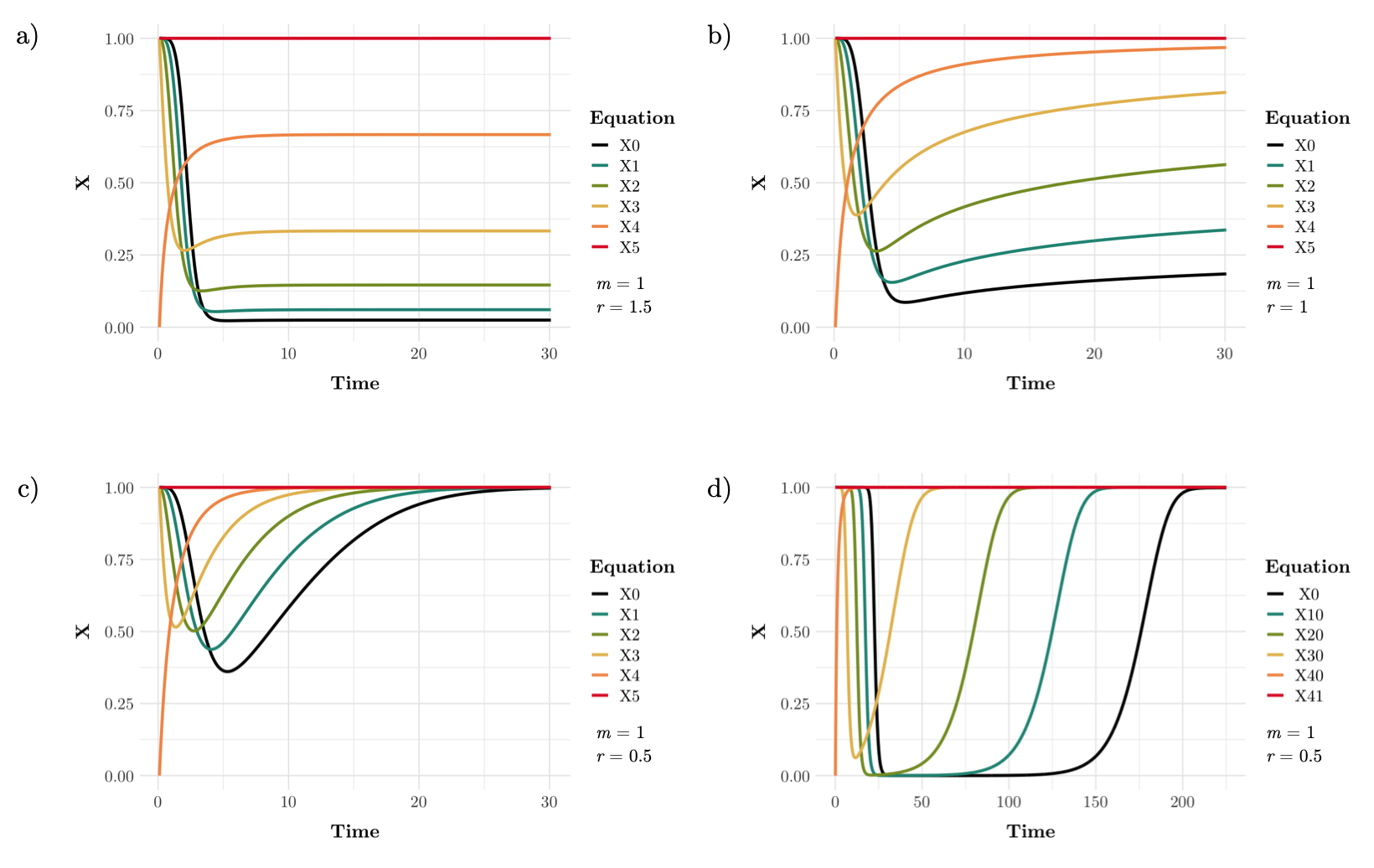}\caption{\small In these pictures we observe the process $X_t$ with $V=\N$, $R_{v,v}=r\delta_0$, $M_{v,v+1}=m\delta_0$ and all other parameters being zero. This is the dual of a branching random walk in which particles branch at rate 1 and migrate from state $v$ to $v+1$ at rate $m$.  In a), b) and c) the starting condition is $(1,1,...)-e_4$ and in d) it is $(1,1,...)-e_{40}$. As observed in Equation \eqref{occupancy}, the black line $(X_0)$ is the graph of the probability that there is no particle at position 4 (resp.\ 40) at time $t>0$, starting the branching random walk with one particle at position zero at time zero.}\end{figure}

Probably, the most important technique used in the study of the PAM is the following assertion.
\begin{proposition}[Feynman-Kac formula]\label{thm:feynmankac}
Let $(Y_t)_{t \geq 0}$ be a simple symmetric random walk in $[K]^d.$ Under the moment condition 
\[ \E \Big[ \big(\frac{\max \{ \xi_v, 2 \}}{\log(\max \{ \xi_v,2 \})}\big)^d \Big]<\infty, \qquad \forall v \in V, \numberthis\label{PAMmomentcond} \]
we have that
\[ f(t,v)=\E[\mathrm{e}^{\int_0^t \xi_{Y_s}\dd s}\mathds 1_{\{Y_t=v\}}]=\E[\mathrm{e}^{\int_0^t \xi_{Y_s}^+\dd s}\mathds 1_{\{Y_t=v\}} \mathds 1_{\{ t < \mathcal T \}}], \qquad t>0, v \in [K]^d \]
is a solution to the PAM, where $(M_t)_{t \geq 0}$ is a Poisson process that is independent of $(Z_t)_{t \geq 0}$, and
\[ \mathcal T = \inf \{ t \geq 0 \colon M_{\int_0^t \xi_{Y_s}^- \dd s} \geq 1 \}. \numberthis\label{Tdef} \] 
 \end{proposition}
The original proof is analytic and can be found in \cite{GM}.  Our construction provides a straightforward proof of this formula using the \lq lonely walker representation\rq~(see Remark 2.7 of \cite{KPAM}).
\begin{proof}
Since in the PAM there is no coalescence, Lemma~\ref{lem:expectationequality} provides a straightforward way to compute $\E[Z_t^{(v)}]$, $v \in V$. Let us first consider the case without death. Then, thanks to the lemma and the uniqueness of the solution of the PAM under \eqref{PAMmomentcond} (cf.~\cite[Theorem 1.2]{KPAM}), for $t \geq 0$, $Z_t^{(v)}$ has the same expectation as $Z'^{(v)}_t$ where the \lq fully coordinated PAM\rq~process $(Z'_t)_{t \geq 0}$ is such that $\Lambda_v=D_v=0$, further, $R_{uv}$ and $M_{vu}$ are replaced by their total masses times $\delta_1$. In the process $(Z'_t)_{t \geq 0}$, all individuals move together and reproduce simultaneously according to a Poisson process $(N_t)_{t \geq 0}$ time-changed by $(\xi^+_v)_{v \in V}$ evaluated along the random walk path $(Y_t)_{t \geq 0}$.
To be more precise, let us define $\tau=\int_0^t \xi^+_{Y_s} d s$. Then
$(Z'_t)_{t \geq 0}=(Z'^{(v)}_{t})_{t \geq 0,v \in V}$ is defined as
\[ Z'^{(v)}_t = 2^{N_{\tau}} \mathds 1_{\{ Y_t = v \}}. \]
Using the probability generating function of a Poisson random variable, one computes
\[ \E\big[Z_t^{(v)}\big]=\E\big[Z'_t \mathds 1_{\{Y_t = v \}}\big] = \E\big[ 2^{N_\tau} \mathds 1_{\{Y_t = v \}} \big] = \E \big[ \E \big[ 2^{N_\tau} \big| \sigma\big(\bar \xi^+, (Y_s)_{0 < s \leq t}\big) \big] \mathds 1_{\{ Y_t = v \}} \big] = \E \big[ \mathrm e^\tau \mathds 1_{\{ Y_t = v \}} \big], \]
which finishes the proof. 

In case there is also death in the model, in the fully coordinated process all individuals simultaneously die 
after the first arrival time of a Poisson process $(M_t)_{t \geq 0}$ independent of $(N_t)_{t \geq 0}$ time-changed by $(\xi^-_v)_{v \in V}$. To be more precise, for $t>0$,
\[ \{ Z'_t = 0 \} = \{ t \geq \mathcal T \}, \]
where $\mathcal  T$ is defined according to \eqref{Tdef}.
Thus, we have
\[ 
\begin{aligned}
  \E\big[Z_t^{(v)}\big] & = \E \big[ \mathrm e^\tau \mathds 1_{\{ Y_t=v \}} \mathds 1_{\{ t< \mathcal T \}} \big]=\E \big[ \mathrm e^{\int_0^t \xi^+_{Y_s} d s} \mathds 1_{\{ Y_t=v \}} \mathds 1_{\{ t< \mathcal T \}} \big]=\E \big[ \mathrm e^{\int_0^t (\xi^+_{Y_s}- \xi^-_{Y_s}) d s}  \mathds 1_{\{ Y_t=v \}}\big] 
  \\ & =\E \big[ \mathrm e^{\int_0^t \xi_{Y_s} d s}  \mathds 1_{\{ Y_t=v \}}\big]. 
\end{aligned}\]
\end{proof}
\begin{example}\label{ex:intermediatePAM}
Since the PAM has a unique solution under the condition \eqref{PAMmomentcond}, there is an uncountable family of coordinated processes $(Z''_t)_{t \geq 0}$ such that $\E[Z''^{(v)}_t]$ equals $\E[Z_t^{(v)}]=\E[Z'^{(v)}_t]$ from the proof of Proposition~\ref{thm:feynmankac}. For example, this is the case for $(Z''_t)_{t \geq 0}$ where the birth and the death rates are the same as for the branching process $(Z_t)_{t \geq 0}$, further, $\Lambda_v=0$ and $M_{uv}=m_{uv}\delta_{\frac{1}{2}}$. In the process $(Z''_t)_{t \geq 0}$, for any $(u,v) \in V \times V$ with $u \neq v$, migration events from $u$ to $v$ happen according to a homogeneous Poisson process, independently of all the other pairs of vertices, and at a migration event each individual situated at $u$ migrates to $v$ independently with probability $1/2$.
\end{example}

\section{Coordination and variance}\label{sec-Var}
In this section, we further analyse the processes that turn out to have the same expectation thanks to Lemma~\ref{lem:expectationequality}. In the case of spatial branching processes with migration, i.e., in the case where there is no coalescence but reproduction, death and migration are possibly present in the model, we compute the variance of the processes. We show that given the total masses of the reproduction, death and migration measures, the variance is maximal in the completely coordinated case and minimal in the independent case.

We say that the collection of measures 
\[ \big\{ (D_v)_{v \in V}, (R_{vu})_{v,u \in V}, (M_{vu})_{v,u \in V}, (\Lambda_v)_{v\in V} \big\} \]
is of type 
\[ \big\{ (d_v)_{v \in V}, (r_{vu})_{v,u \in V}, (m_{vu})_{v,u \in V}, (c_v)_{v\in V} \big\} \]
if $D_v[0,1]=d_v$, $R_{vu}[0,1]=r_{vu}$, $M_{vu}[0,1]=m_{vu}$ and $\Lambda_v[0,1]=c_v$. In case the structured branching coalescing process $(Z_t)_{t \geq 0}$ (defined according to Definition~\ref{def:random_rate}) has parameters of this type, we write $(Z_t)_{t \geq 0} \in \mathcal{K}((d_v)_{v}, (r_{vu})_{u,v}, (m_{vu})_{u,v}, (c_v)_{v} )$.

\begin{lemma}\label{lem:varianceequality}
Let  $d_v, r_{vu},m_{vu}\geq 0$ and $c_v=0$ for all $u,v \in V$. Then, \[ \sup_{(Z_t)_{t \geq 0} \in \mathcal{K}((d_v)_{v}, (r_{vu})_{u,v}, (m_{vu})_{u,v}, (0)_v )}\mathrm{Var}[Z^{(v)}_s]=\mathrm{Var}[\bar Z^{(v)}_s] \]  for all $v \in V$, where $\bar Z^{(v)}_s$ is such that for all $u,w \in V$, $M_{u,w}=m_{u,w}\delta_1$, $R_{uw}=r_{uw}\delta_1$ and $D_{u}=d_u\delta_1$, and  \[ \inf_{(Z_t)_{t \geq 0}\in \mathcal{K}((d_v)_{v}, (r_{vu})_{u,v}, (m_{vu})_{u,v}, (0)_v )}\mathrm{Var}[Z^{(v)}_s]=\mathrm{Var}[\underline Z^{(v)}_s] \]  for all $v \in V$, where $\underline Z^{(v)}_s$ is such that for all $u,w \in V$, $M_{u,w}=m_{u,w}\delta_0$, $R_{uw}=r_{uw}\delta_0$ and $D_{u}=d_u\delta_0$. Here, $(0)_v$ denotes the collection of $|V|$ instances of the zero measure.
\end{lemma} 
\begin{proof}
Applying the form \eqref{eq:generator} of the generator for $f_v(z)=z_v^2$, we obtain, using Dynkin's formula,
\begin{align*}
    \mathbb{E}\Big[{Z_t^{(v)}}^2\Big]-{Z_0^{(v)}}^2& = \E \Big[ \int_0^t \Big(\sum_{u \in V}  \int_0^1\sum_{i=1}^{Z_s^{(v)}} (-2 Z_s^{(v)}i + i^2) \binom{Z_s^{(v)}}{i} y^i (1-y)^{Z^{(v)}_s-i} \frac{1}{y} M_{vu}(\dd y)
   \\ & \qquad + \sum_{u \in V}  \int_0^1  \sum_{i=1}^{Z_s^{(u)}} (2 Z_s^{(u)}i + i^2) \binom{Z_s^{(u)}}{i} y^i (1-y)^{Z^{(u)}_s-i} \frac{1}{y} M_{uv}(\dd y)\\
    & \qquad +    \int_0^1 \sum_{i=1}^{Z_v^{(s)}}(-2 Z_s^{(v)} i+i^2) \binom{Z_s^{(v)}}{i} y^i (1-y)^{Z_s^{(v)}-i} \frac{1}{y} D_v(\dd y) \\
    & \qquad + \sum_{u \in V}  \int_0^1 \sum_{i=1}^{Z_s^{(u)}} (2 Z_s^{(u)} i+i^2)\binom{Z_s^{(u)}}{i} y^i (1-y)^{Z_s^{(u)}-i} \frac{1}{y} R_{uv}(\dd y)
   \Big) \dd s \Big]. \numberthis\label{eq:lastline}
\end{align*}
Thus, recalling that for a binomial random variable with parameters $n,p$ one has $\E[X^2]=np(1-p+np)$ and using \eqref{eq:expectations} together with the Fubini-Tonelli theorem, we can interchange the outermost expectation with the integration from 0 to $t$ on the right-hand side of \eqref{eq:lastline} by the same arguments as in the proof of Lemma \ref{lem:expectationequality} and write the equation in differential form as follows 
\[ \numberthis\label{eq:Dynkinsquare} 
\begin{aligned}
  \frac{\dd}{\dd t}  \mathbb{E}& \Big[ {Z_t^{(v)}}^2 \Big]=  \E \big[ \sum_{u \in V}  \int_0^1 Z_s^{(u)} (1-y+Z_s^{(u)} y-2Z_s^{(u)}) \big]  M_{vu}(\dd y)
   \\ & \qquad +  \E \big[\sum_{u \in V}  \int_0^1  Z_s^{(v)} (1-y+Z_s^{(v)} y+2 Z_s^{(v)}) \big] M_{uv}(\dd y)
 \\ & \qquad +  \E \big[\int_0^1  Z_s^{(v)} (1-y+Z_s^{(v)} y-2 Z_s^{(v)}) \big] D_v(\dd y) \\
    & \qquad +  \E \big[\sum_{u \in V} \int_0^1 Z_s^{(u)}(1-y+Z_s^{(u)}y+2 Z_s^{(u)})\big] R_{uv}(\dd y).
\end{aligned}
\]
Now, $\mathrm{Var}[Z_t^{(v)}] = \E[{Z_t^{(v)}}^2]-\E[Z_t^{(v)}]^2 $, and  Lemma~\ref{lem:expectationequality} implies that $\E[Z_t^{(v)}]^2$ is constant given the total masses of all migration, death and reproduction measures. Hence, in order to maximize (minimize) $\mathrm{Var}[Z_t^{(v)}]$ given these total masses, it suffices to maximize (minimize) the right-hand side of \eqref{eq:Dynkinsquare}. Note that for all $v \in V$, $Z_s^{(v)}$ takes nonnegative integer values, and given that it is zero, $Z_s^{(v)}(1-y+Z_s^{(u)} y \pm 2Z_s^{(u)})=0$. Further, for $Z_s^{(v)} \geq 1$, $Z_s^{(v)}y \geq y$. It follows that given the total masses, any term on the right-hand side of \eqref{eq:Dynkinsquare} is maximal for the corresponding measure being a constant multiple of $\delta_1$ and minimal for the measure being a constant multiple of $\delta_0$. Hence,
we conclude the lemma.
\end{proof} 

The previous result allows us to bound the variance of all the processes whose expectation solves the PAM.
\begin{corollary}
Assume that $(Z_t)_{t \geq 0}$ is a coordinated branching process such that $\E[Z_t^{(v)}]$ is a solution of equation \eqref{PAM}. Let $(Y_t)_{t \geq 0}$ be a simple symmetric random walk in $[K]^d.$ Then, recalling $\xi_v=\xi_v^+-\xi_v^-$, for $v\in V$,
\begin{align*}
{\rm Var}\big[Z_t^{(v)}\big]\leq &\E\Big[\exp\Big(\int_0^t \xi^+_{Y_s}\dd s\Big)\Big(\exp \big(2\int_0^t \xi^+_{Y_s}\dd s\big)-1\Big)\mathds 1_{\{Y_t=v\}} \mathds 1_{\{ t < \mathcal T \}} \Big]\\
=&\E\Big[\exp\Big(\int_0^t \xi_{Y_s}\dd s\Big)\Big(\exp\big(2\int_0^t \xi^+_{Y_s}\dd s\big)-1\Big)\mathds 1_{\{Y_t=v\}} \Big].
\end{align*}
Here, $(M_t)_{t \geq 0}$ is a Poisson process independent of $(Z_t)_{t \geq 0}$, and $\mathcal T=\inf \{ t \geq 0 \colon M_{\int_0^t \xi^{-}_{Y_s} \dd s} \geq 1 \}$. 
\end{corollary}
\begin{proof}
The right-hand side of the equation to be proven is the variance of $\bar Z_t^{(v)}$ in the notation of Lemma \ref{lem:varianceequality}, and thus the statement follows directly from this lemma. To calculate the variance we compute the second moment as follows:
\begin{align*}
\E\big[(Z_t^{(v)})^2\big]
= &\E\big[ 4^{N_{\int_0^t \xi_{Y_s}\dd s}}\mathds 1_{\{Y_t=v\}} \mathds 1_{\{ t < \mathcal T \}}  \big]=\E \big[ \mathrm{e}^{\ln(4)N_{\int_0^t \xi_{Y_s}\dd s}}\mathds 1_{\{Y_t=v\}} \mathds 1_{\{ t < \mathcal T \}}  \big] \\
= &\E\big[\mathrm{e}^{3\int_0^t \xi_{Y_s}\dd s}\mathds 1_{\{Y_t=v\}} \mathds 1_{\{ t < \mathcal T \}} \big] .
\end{align*}
where in the last equality we used the formula for moment generating function of a Poisson random variable. As shown in Proposition \ref{thm:feynmankac}, $\E[\bar Z_t^{(v)}]^2=\E[\mathrm{e}^{2\int_0^t \xi_{Y_s}\dd s}\mathds 1_{\{Y_t=v\}} \mathds 1_{\{ t < \mathcal T \}}]$, which together with the definition of $\mathcal T$ implies
$$
{\rm Var}\big[\bar Z_t^{(v)}\big]=\E \big[\mathrm{e}^{2\int_0^t \xi_{Y_s}\dd s}(\mathrm{e}^{\int_0^t \xi_{Y_s}\dd s} -1) \mathds 1_{\{Y_t=v\}} \mathds 1_{ \{ t < \mathcal T \} } \big]=\E\big[\mathrm{e}^{\int_0^t \xi_{Y_s}\dd s}(\mathrm{e}^{2\int_0^t \xi^+_{Y_s}\dd s}-1)\mathds 1_{\{Y_t=v\}} \big].
$$
\end{proof}

\section{Extensions to infinite graphs}\label{sec-infinity}
The main results of the present paper tell about the case when $G=(V,E)$ is a finite graph (where we recall that this graph was defined in Section~\ref{sec-modeldef}). Let us now discuss under what conditions these statements can be extended to an infinite graph in general. 

Let $G=(V,E)$ be an infinite, connected, locally finite graph. Choose a collection 
\[
\mathfrak M=\{ R_{uv},D_w, M_{uv}, \Lambda_w  \colon  w \in V, (u,v) \in E \} 
\] 
of elements of $\mathcal M[0,1]$ interpreted similarly to Definition~\ref{def:random_rate} for the case of a finite graph, and a collection
\[
\mathfrak P=\{ N^{R_{uv}}, N^{D_w}, N^{M_{uv}}, N^{\Lambda_w} \colon w \in V, (u,v) \in E \}
\]
of independent Poisson point processes, defined analogously to the case of a finite graph (cf.~page~\pageref{PPP's}), involving the measures contained in $\mathfrak M$.

Let us fix $v_0 \in V$. For $v,w \in V$ let $d(v,w)$ denote the graph distance of $v$ and $w$, i.e., the length of the shortest path of edges connecting $v$ and $w$ in the graph. Then for $N \in \N_0$ we define
\[ V^N=\{ v \in V \colon d(v,v_0) \leq N \} \]
as the set of vertices situated at graph distance at most $N$ from $v_0$. We also put $V_{-1}=E_{-1}=\emptyset$. We denote by $G^N=(V^N,E^N)$ the subgraph of $G$ spanned by $V^N$. Then we let $(Z_{N,t})_{t \geq 0}=(Z^{(v)}_{N,t})_{t \geq 0,v \in V}$ to be the process defined according to the Poisson point process representation for finite graphs involving the measures contained in the set
\[
\mathfrak M^N=\{ R_{uv},D_v, M_{uv}, \Lambda_v  \colon  w \in V^N, (u,v) \in E^N \}
\]
and the corresponding Poisson point processes included in the set
\[
\mathfrak P=\{ N^{R_{uv}}, N^{D_w}, N^{M_{uv}}, N^{\Lambda_w} \colon w \in V^N, (u,v) \in E^N \}.
\]
Now, for $N\in \N_0$ we let
\[ \tau_N = \inf \{ t \geq 0 \colon \exists v \in V^N \setminus V^{N-1} \colon Z_{N,t}^{(v)} >0 \}. \]
The crucial assumption in order to define a limiting process on the infinite graph is the following nonexplosion condition:
\[ \lim_{N \to \infty} \tau_{N} = \infty, \numberthis\label{nastycondition} \]
almost surely given the initial condition $Z_0^{(v)}=\mathds 1_{\{ v=v_0\}}$, $v \in V$. 
Indeed, then for $t \geq 0$, the random variable
\[ N_t = \inf \{ N \in \N_0 \colon t < \tau_N \}. \]
is almost surely finite. Hence, if we define
\[ Z_t = \lim_{N \to \infty} Z_{N,t \wedge \tau_N}, \]
then $Z_t=Z_{N_t,t}$, almost surely, 
and for all $t \geq 0$, $Z_{N_t,t}$ is almost surely well-defined according to the Poissonian construction for finite graphs.

Condition  \eqref{nastycondition} is difficult to check. Despite it should be true in many interesting cases, it is not hard to come up with examples in which even if all the measures have a bounded mass, condition  \eqref{nastycondition} is not satisfied. It remains an open question to find easy to verify sufficient conditions to extend the results presented in this paper to infinite graphs.

Lemma \ref{lem:expectationequality} can be extended to the infinite case under assumption \eqref{nastycondition}.
\begin{corollary}\label{cor:Emonconv}
Assume that for every $N$, $(Z_t^N)_{t \geq 0}$ as constructed above fulfills the condition of Lemma \ref{lem:expectationequality} and the assumption \eqref{nastycondition} is true. Then, for all $t \geq 0$ 
$$
\E[Z_t]=\lim_{N\rightarrow \infty} \E[Z_{t \wedge \tau_N}^N].
$$
\end{corollary}
\begin{proof}
After observing that for every $t>0$, $Z_{t \wedge \tau_N}^N$ is an increasing function of $N$, the proof follows from monotone convergence.
\end{proof}

Next, we recall three classical processes that are strongly related to each other, two of them belonging to our class of processes: one independent one (a branching random walk) and one fully coordinated one (the binary contact path process), and the third one being the contact process. These processes are usually studied on infinite graphs such as $\mathbb Z^d$ or uniform trees (see~\cite{Liggett} for details). 
\begin{example}[Contact process, binary contact path process and branching random walk]
Let $D,R>0$, let $G=(V,E)$ be a (possibly infinite) graph and let $(\bar Z_t)_{t \geq 0}=(Z_t^{(v)})_{v\in V, t\geq 0}$ be the $\N_0^{|V|}$-valued Markov process with transitions
\begin{equation}
z \mapsto 
\begin{cases}
z- z_ve_v, & \textrm{ at rate } D,\, v\in V,\\
z+ z_ve_u, & \textrm{ at rate }R \mathds 1_{(v,u)\in E},\, u,v\in V.
\end{cases}
\end{equation}
It is clear that the process $(\bar Z_t)_{t \geq 0}$ satisfies Condition \eqref{nastycondition}. It is called the \emph{binary contact path process}, and it was first studied in \cite{Griffeath}. Now consider $\bar C_t=(C_t^{(v)})_{v\in V}=(\mathds 1_{Z_t^{(v)}>0})_{v\in V}$, $t\geq 0$ and observe that $(C_t)_{t \geq 0}$ is the contact process on the graph $G$ with parameters $D$ and $R$ (cf.~\cite[Section 2]{BG}). Let further $(\bar N_t)_{t \geq 0}=(N_t^{(v)})_{v\in V,t\geq 0}$ be the $\N_0^{|V|}$-valued Markov process with transitions
\begin{equation}
z \mapsto 
\begin{cases}
z- e_v, & \textrm{ at rate }  D z_v,\, v\in V,\\
z+ e_u, & \textrm{ at rate } R z_v \mathds 1_{(v,u)\in E} ,\, u,v\in V.
\end{cases}
\end{equation}
The process $(\bar N_t)_{t \geq 0}=(N_t^{(v)})_{v\in V,t\geq 0}$ is the \emph{branching random walk} associated to the contact process. We observe that by Lemma~\ref{lem:expectationequality},
\[ \E_{\bar z}[ Z_t^{(v)}]=\E_{\bar z}[N_t^{(v)}], \qquad \forall v \in V, t \geq 0, \bar z \in \N_0^{|V|}, \numberthis\label{contactexpectation}\] 
and thus, $(\bar N_t)_{t \geq 0}$ also satisfies Condition \eqref{nastycondition}. 

It is easy to see that if $C_0^{(v)} \leq N_0^{(v)}$ for all $v \in V$, then $C_t^{(v)} \leq N_t^{(v)}$ for all $t \geq 0$ and $v \in V$; in this sense, the contact process can be seen as a branching random walk where particles at the same site (vertex) coalesce. These assertions can be found in \cite[page 32]{Liggett}. A well-known consequence \cite[page 43]{Liggett} of this comparison is that if all degrees of the graph $G$ are bounded by $d \in \N$, then $\lim_{t \to \infty} |C_t|=0$ holds almost surely whenever the branching random walk is subcritical, i.e., $dR-D < 0$. Further, in the critical case $dR=D$, since particles of this branching random walk can actually die with positive probability, the branching random walk and hence also the contact process dies out. 

\end{example}

We now present two applications of the results of the present paper that extend to infinite graphs thanks to Corollary~\ref{cor:Emonconv}. To start with, the proof that we provided for the Feynman-Kac formula (Proposition~\ref{thm:feynmankac}) remains true for infinite graphs satisfying \eqref{nastycondition}. Using classical results on the PAM, we can provide a more explicit sufficient condition for \eqref{nastycondition} as follows.
\begin{corollary}
Proposition~\ref{thm:feynmankac} remains true for $V=\mathbb Z^d$, $d \in \N$ (instead of $V=[K]^d$ for $K \in \N$), in case the moment condition~\eqref{PAMmomentcond} holds.
\end{corollary}
\begin{proof}
As already mentioned, according to \cite[Theorem 1.2]{KPAM}, \eqref{PAM} has a unique solution in $\mathbb Z^d$ given that \eqref{PAMmomentcond} holds. Using Corollary~\ref{cor:Emonconv}, we conclude that $f(t,v)= \E[Z_t^{(v)}]$ equals this solution. Finally, thanks to Proposition~\ref{thm:feynmankac} and monotone convergence, this solution must be equal to $f(t,v)=\E[\mathrm{e}^{\int_0^t \xi_{Y_s}^+\dd s}\mathds 1_{\{Y_t=v\}} \mathds 1_{\{ t < \mathcal T \}}]$. This implies the corollary.
\end{proof}

Finally, as an additional application of the invariance of expectation on infinite graphs, we consider a branching random walk on a $d$-uniform rooted tree, and we provide a probabilistic interpretation of its expectation process in terms of the underlying ``fully coordinated'' process.
\begin{remark}\label{ex:tree}
Let the graph $G=(V,E)$ be a $d$-uniform rooted tree for some $d \in \N$. That is, there is a distinguished vertex $o$ called the root, which is the only vertex in the 0th generation of the vertex set $V$, vertex generations are pairwise disjoint, and for $n \in \N_0$, each vertex in generation $n$ is connected by an edge to precisely $d$ vertices in generation $n+1$, so that each vertex in generation $n+1$ has precisely one neighbour from generation $n$. For each $n \in \N_0$, let us fix an arbitrary indexing $v_{(n,1)},\ldots, v_{(n,d^n)}$ of the vertices of generation $n$, in particular, $v_{(0,1)}=o$. Then we fix $r,\mu>0$ and define a branching random walk, i.e., a particle system $(Z_t)_{t \geq 0}$ on $G$ according to Definition~\ref{def:random_rate} with the following rates: $\Lambda_v = D_v=0$ for all $v \in V$, $R_{vv}=r \delta_0$ for all $v\in V$ and $R_{vu} = 0$ for all $u,v \in V, u \neq v$, further, $M_{vu}=\frac{1}{d} \mu \delta_0$ in case $(v,u) \in E$ and $u$ belongs to one generation higher than $v$ and $M_{vu}=0$ otherwise. In words, particles create an offspring at rate $r$ and jump to a uniformly chosen neighbouring vertex in the next generation at rate $\mu$, independently of all the other particles. 

Let $(\bar Z_t)_{t \geq 0}$ be the associated fully coordinated process, i.e., the process that exhibits also no coalescence or death, and whose reproduction and migration measures $\bar R_{vu},\bar M_{vu}, u,v\in V$ are obtained by replacing $\delta_0$ with $\delta_1$ in the definition of the corresponding measure $R_{vu}$ respectively $M_{vu}$ of $(Z_t)_{t \geq 0}$. In this process, all particles move simultaneously at rate $\mu$ to one of the neighbours of the present vertex in the next generation, and at rate $r$ they reproduce simultaneously. In order to simplify the notation, we use the notations $Z_t^{(k,i)}$ and $\bar Z_t^{(k,i)}$ instead of $Z_t^{(v_{(k,i)})}$ resp.\ $\bar Z_t^{(v_{(k,i)})}$ for the coordinates of $Z_t$ resp.\ $\bar Z_t$ corresponding to the vertex $v_{(k,i)}$, $k \in \N_0$, $i=1,\ldots, d^k$. For $k \in \N$, $j \in \{0,1,\ldots,k-1\}$ and $i=1,\ldots,d^k$, let us write $a(j,k,i)$ for the ancestor of $(k,i)$ in generation $j$, i.e., for the unique vertex of the form $(j,\cdot)$ connected by a path (i.e., sequence of edges) to $(k,i)$. In particular, $a(0,k,i)=o$ for all $i=1,\ldots,d^k$.

Then, it is easy to check that the nonexplosion condition~\eqref{nastycondition} is satisfied, and hence Corollary~\ref{cor:Emonconv} is applicable. Together with Lemma~\ref{lem:expectationequality}, we obtain that starting from one single particle at $o$ at time zero, for $t\geq 0$, $k \in \N_0$ and $i=1,\ldots,d^k$ we have
\begin{align*}
    \E\big[ Z_t^{(k,i)} \big] = \E\big[ \bar Z_t^{(k,i)} \big] = \mathrm e^{rt} \frac{(t\mu)^k}{k!} \mathrm e^{-t\mu}\frac{1}{d^k} = \mathrm e^{t(r-\mu)} \big(\frac{t\mu}{d}\big)^k \frac{1}{k!}. \numberthis\label{eqexpbranching}
\end{align*}
This yields a probabilistic interpretation for the sytem of ODEs
\[ \begin{aligned} \frac{\dd}{\dd t} f(t,(k,i)) &= (r-\mu) f(t,(k,i)) + \frac{\mu}{d} f(t,a(k-1,k,i)), \quad k \in \N, i=1,\ldots,d^k, \\
\frac{\dd}{\dd t} f(t,o) & = (r-\mu) f(t,o), \\
f(0,(k,i))&=\delta_{0,k}, \end{aligned} \numberthis\label{reactiondiffusion} \]
where it is straightforward to derive that its unique solution $f(t,(k,i))$ equals $\E[Z_t^{(k,i)}]$ from \eqref{eqexpbranching}. 
Note that the representation \eqref{eqexpbranching} of the solution of the system~\eqref{reactiondiffusion} by a fully coordinated process is in analogy to the one provided for the solution of the PAM in the proof of Proposition~\ref{thm:feynmankac}. Hence, \eqref{eqexpbranching} can also be interpreted as an explicit Feynman--Kac representation. 

By Lemma~\ref{lem:expectationequality}, for all structured branching-coalescing processes with coordination on $G$ where all the measures of the form $R_{vv}$ and $M_{vu}$ have the same total mass as for the branching random walk and all other measures are zero, the expectation of the process solves \eqref{reactiondiffusion}.
\end{remark}

\section*{Acknowledgements}
We thank J.~Blath, F.~Hermann, M.~Wilke-Berenguer and D.~Valesin for interesting discussions and comments. We are grateful for Argelia Hernandez Recio and Julio Nava Trejo for producing Figure~\ref{imagen}. We also thank two anonymous reviewers for insightful comments and suggestions. AGC was supported by the grant CONACYT CIENCIA BASICA A1-S-14615. NK was supported by the DFG Priority Programme SPP 1590 ``Probabilistic Structures in Evolution'', grant no.~KU 2886/1-1. AT was supported by the DFG Priority Programme SPP 1590 ``Probabilistic Structures in Evolution''.

\end{document}